\documentclass[a4paper,10pt]{amsart}

\pdfoutput=1

\usepackage[text={420pt,660pt},centering]{geometry}
 \usepackage{cancel}
\usepackage{esint,amssymb} 
\usepackage{graphicx}

\usepackage{MnSymbol}
\usepackage{mathtools}
\usepackage[colorlinks=true, pdfstartview=FitV, linkcolor=blue, citecolor=blue, urlcolor=blue,pagebackref=false]{hyperref}
\usepackage{bm}
\usepackage{mathrsfs}
\usepackage{xcolor}

\usepackage{accents} \usepackage{enumitem}
\setlist{leftmargin=0.5\leftmargin}

\definecolor{darkgreen}{rgb}{0,0.5,0}
\definecolor{darkblue}{rgb}{0,0,0.7}
\definecolor{darkred}{rgb}{0.9,0.1,0.1}

\newtheorem{proposition}{Proposition}
\newtheorem{theorem}[proposition]{Theorem}
\newtheorem{lemma}[proposition]{Lemma}

\theoremstyle{remark}
\newtheorem{remark}[proposition]{Remark}

\theoremstyle{definition}

\numberwithin{equation}{section}
\numberwithin{proposition}{section}
\numberwithin{figure}{section}
\numberwithin{table}{section}

\newcommand{\Z}{\mathbb{Z}}
\newcommand{\N}{\mathbb{N}}

\newcommand{\R}{\mathbb{R}}

\newcommand{\E}{\mathbb{E}}

\newcommand{\eps}{\varepsilon}

\renewcommand{\le}{\leqslant}
\renewcommand{\ge}{\geqslant}
\renewcommand{\leq}{\leqslant}
\renewcommand{\geq}{\geqslant}

\renewcommand{\subset}{\subseteq}

\renewcommand{\hat}{\widehat}

\newcommand{\Ll}{\left}
\newcommand{\Rr}{\right}
\renewcommand{\d}{\mathrm{d}}

\newcommand{\1}{\mathbf{1}}
\newcommand{\rz}{z}
\newcommand{\mcl}{\mathcal}

\DeclareMathOperator{\supp}{supp}

\newcommand{\la}{\left\langle}
\newcommand{\ra}{\right\rangle}

\renewcommand{\S}{\mathsf{S}}

\newcommand{\comple}{{\complement}}

\newcommand{\one}{\mathbf{1}}

\newcommand{\lpcfg}{\mcl L}
\newcommand{\lo}{\ell}

\newcommand{\itr}{\mathrm{int}}
\newcommand{\llbracket}{[\!\![}
\newcommand{\rrbracket}{] \!\! ]}
\newcommand{\pa}{\mathsf{P}}
\newcommand{\rcE}{\phi}
\newcommand{\rcP}{\phi}
  \newcommand{\bc}{\mathbf{c}}
\newcommand{\calL}{\mathcal{L}}
\newcommand{\dst}{d} \newcommand{\crit}{\mathrm{c}}

\DeclareSymbolFont{cmrboldroman}{OT1}{cmr}{bx}{n}
\SetSymbolFont{cmrboldroman}{bold}{OT1}{cmr}{bx}{n}
\DeclareSymbolFontAlphabet{\mathbf}{cmrboldroman}

\begin{document}

\author[H.-B. Chen]{Hong-Bin Chen}
\address[H.-B. Chen]{NYU-ECNU Institute of Mathematical Sciences, NYU Shanghai, Shanghai, China}

\author[J. Xia]{Jiaming Xia}

\address[J. Xia]{
Center for Mathematics and Interdisciplinary Sciences,
Fudan University (Primary) \& Shanghai Institute for Mathematics and Interdisciplinary Sciences (SIMIS) (Secondary),
Shanghai 200433, China
}

\subjclass[2010]{}
\date{\today}

\title[Multi-point functions of a fuzzy $4$-Potts model]{Multi-point functions of a full-plane two-state\\ fuzzy $4$-Potts model}

\begin{abstract}
We study the full-plane two-state fuzzy $4$-Potts model, obtained by assigning independent balanced $\{\pm1\}$ spins to the open clusters of a critical $q=4$ random-cluster configuration. Its spin field has the same law as a single-spin projection and as the polarization field of the isotropic Ashkin--Teller model at its Potts point. We prove that, after proper normalization, all even multi-point spin correlation functions converge to explicit conformally covariant Coulomb-gas type neutral charge sums. As a consequence, we prove convergence in law of the rescaled magnetization field and identify its limit in law with $2\mathopen{:}\cos(X/2)\mathclose{:}$ for a full-plane Gaussian free field $X$. The proof combines the Baxter--Kelland--Wu coupling, convergence of the six-vertex height function to the Gaussian free field, and a charge-completion mechanism: an enlarged discrete sum over charge assignments with total charge in $4\mathbb Z$ produces a combinatorial cancellation of connection patterns, while only the neutral charge sector survives in the scaling limit.
\end{abstract}

\maketitle

\section{Background and motivation}

An object of primary interest in two-dimensional lattice models and continuum statistical mechanics is the collection of \emph{multi-point correlation functions}. These functions encode the joint behavior of local observables, such as spins or cluster connectivities, at $n$ prescribed locations. Their scaling limits are expected to describe the operator algebraic structure and universal symmetries predicted by conformal field theory (CFT). Even for classical integrable models, obtaining explicit formulas for general multi-point functions is notoriously difficult. For instance, in the planar Ising model, while the two-point function is classical, deriving exact multi-point functions requires sophisticated techniques from discrete complex analysis and Pfaffian identities \cite{Onsager44, WuMcCoyTracyBarouch76, Smirnov10, CHI2015, chelkak2021correlations}. In other paradigms, including Potts models, random cluster models, and percolation, the full multi-point structure is typically known only partially through scaling exponents or asymptotic descriptions from SLE/CLE theory \cite{LSW2004, Beffara08, MillerSheffield16I, MillerSheffield16II, MillerSheffield16III}; truly \emph{explicit formulas} for these models remain exceedingly rare.

Recent breakthroughs illustrate both the depth and the scarcity of such results. For critical percolation on the triangular lattice, it was recently established that multi-point functions of an associated spin field possess conformally covariant scaling limits, identifying logarithmic singularities in the process \cite{camia2024conformal, camia2024conformally, camia2024logarithmic}. These results rely crucially on Smirnov’s proof of conformal invariance for critical site percolation and the CLE$_6$ description of the cluster boundaries. For the FK--Ising model ($q=2$), renormalized multi-point connectivity probabilities have been shown to converge to conformally covariant limits in simply connected domains \cite{camia2025}, complementing classical work on Ising correlations \cite{CHI2015} and their extension to all primary fields \cite{chelkak2021correlations}.

In this paper, we consider a two-state fuzzy $4$-Potts model on $\mathbb{Z}^2$. This model is a specific case of fuzzy Potts models, which belong to the broader class of ``divide-and-color'' models. These models are defined by first sampling a random cluster configuration and then independently assigning colors or spins to each cluster. There has been renewed interest in this area. For example, recent work has determined arm exponents for continuum fuzzy Potts models \cite{kohler2025fuzzy} and identified the exact bulk one-arm exponent for colored CLE \cite{liu2024bulk}. Furthermore, an explicit formula for the three-point connectivity constant was recently obtained in the continuum fuzzy Potts setting \cite{cai2025three}, providing a rare closed-form result for a non-trivial multi-point observable.

The specific model we study here can be described as follows: we sample a critical $q=4$ random cluster configuration and assign an independent spin $\S \in \{\pm 1\}$ to each open cluster. Our analysis relies on the Baxter--Kelland--Wu coupling, which connects random cluster models to the six-vertex model. A central input for our proof is a recent result from \cite{GFF6vertex}, which establishes the convergence of the six-vertex height function to the Gaussian free field (GFF). Specifically, for the isotropic six-vertex model with weights $a=b=1$ and $c \in [\sqrt{3}, 2]$ (corresponding to a spectral parameter $\Delta \in [-1, -1/2]$), the associated height function $h$ converges in the scaling limit to a properly scaled full-plane GFF.

By leveraging this GFF convergence, the Baxter--Kelland--Wu coupling, and the fact that the spin correlations of the fuzzy Potts field reduce to cluster connectivity events, we rigorously identify the scaling limit of these even $n$-point functions: after normalization by $(\delta^{\frac{1}{8}+o(1)})^{-n}$, they converge to an explicit algebraic expression (see Theorem~\ref{t}):
\[
\sum_{\substack{\xi\in\{\pm1\}^n\\ \sum\xi_i=0}} \ \prod_{1\le i<j\le n} |x_i-x_j|^{\xi_i\xi_j/4}.
\]
We then use this result to obtain convergence of the magnetization field (see Theorem~\ref{t.Phi}).

The structure of this formula mirrors the Coulomb-gas description of correlation functions in $c=1$ conformal field theories \cite{dotsenko1984conformal}. In this formalism, multi-point functions of vertex operators are decomposed into sums over charge assignments that satisfy a global neutrality condition. Each assignment contributes a product of pairwise factors of the form $|x_i - x_j|^{\alpha_i \alpha_j}$. The expression precisely reflects this structure, with the signs $\xi_i$ acting as discrete ``charges''.

A key point in the proof is a charge-completion and cancellation mechanism (see Lemma~\ref{l.cancel}). Although the limiting Coulomb-gas formula only contains neutral charge assignments, at the discrete level we sum over all assignments with total charge in $4\Z$. This enlarged charge sum cancels all connection patterns with non-constant primal/dual types, leaving exactly the even FK connection patterns that contribute to the fuzzy spin correlations. The non-neutral charge sectors then vanish in the scaling limit, so only the neutral Coulomb-gas terms remain.

The model considered here also corresponds exactly to the Ashkin--Teller model at its isotropic Potts point (see Remark~\ref{r.AT}). Consequently, the fuzzy spin $\S$ can be identified with the projection of the two-component Ashkin--Teller spin $(\sigma, \tau)$ onto a single factor, and has the same law as the polarization spin $\sigma\tau$ at this point. In the physics literature, the multi-point correlation functions of this individual spin have been actively studied along the continuous Ashkin--Teller critical line, most notably through exact lattice mappings \cite{kadanoff1979correlation} and the conformal bootstrap of the four-point function \cite{zamolodchikov1985two}. Thus, our main result can also be viewed as a rigorous derivation of the multi-point function for this projected Ashkin--Teller spin and for the polarization field at the Potts point.

\section{Setting and main results}

Given a finite graph $G=(V,E)$ on $\Z^2$, a percolation configuration $\omega$ on $G$ is an element of $\{0,1\}^E$. An edge $e\in E$ is called \emph{open} if $\omega_e=1$ and \emph{closed} otherwise. A configuration $\omega$ can be seen as a subgraph of $G$ with vertex-set $V$ and edge-set $\{e\in E:\omega_e=1\}$. The \emph{random-cluster measure} on $G$ with edge-weight $p\in(0,1)$, cluster-weight $q>0$, and free boundary condition is defined by
\begin{equation*}
    \rcP^{\mathrm{free}}_{G,p,q}[\omega]
    =\frac{1}{Z^{\mathrm{free}}_{G,p,q}} \left(\frac{p}{1-p}\right)^{|\omega|}q^{k\left(\omega\right)},
\end{equation*}
where $|\omega|=\sum_{e\in E}\omega_e$ is the number of open edges, $k(\omega)$ is the number of connected components of the graph induced by $\omega$, and $Z^{\mathrm{free}}_{G,p,q}$ is the normalizing factor.

Here, we focus on the case $q=4$ and omit it from the notation. Let $\rcP_{\mathbb Z^2,p}$ be the infinite-volume limit of the measures $\rcP^{\mathrm{free}}_{G,p}$ by letting $G$ tend to $\mathbb Z^2$. 
The model undergoes a phase transition at
$ p_\mathrm{c}=\inf\{p\in[0,1]:\rcP_{\mathbb Z^2,p}[0\leftrightarrow \infty]>0\}$.
It was proved in \cite{beffara2012self} that $p_\mathrm{c}=2/3$ and in \cite{DumSidTas13} that the phase transition is continuous. In this paper, we consider the model at its critical point and write $\phi_{\Z^2}=\phi_{\Z^2,p_\mathrm{c}}$ henceforth. For each $\delta>0$, we denote by $\phi_{\delta\Z^2}$ the corresponding measure on the rescaled lattice $\delta\Z^2$.

Given a configuration $\omega$ sampled from $\phi_{\delta\Z^2}$, let $\mathfrak{C}$ be the collection of clusters $\mathcal{C}$ of vertices connected by open edges. Then, sample i.i.d.\ $\{\pm1\}$-valued balanced Rademacher random variables $(\boldsymbol{\sigma}_\mathcal{C})_{\mathcal{C}\in\mathfrak{C}}$. For each $z\in\delta\Z^2$, we define
\begin{align}\label{e.S=}
    \S_z = \boldsymbol{\sigma}_\mathcal{C}\quad\text{if}\quad z\in \mathcal{C}\ \text{ and }\ \mathcal{C}\in \mathfrak{C}.
\end{align}
We denote by $\E_\delta$ the expectation with respect to $\S$ on $\delta\Z^2$.
Notice that this is analogous to the Edwards--Sokal coupling; the difference is that instead of $4$-Potts spins we are assigning Ising spins.

This model $\S$ is a type of fuzzy Potts model, which further belongs to the family of ``divide and color'' models. One can think of $\S$ as mapping two colors in the $4$-Potts model to $+1$ and the other two to $-1$. A similar two-color spin model is also used in~\cite[(1.8)]{camia2024conformal} where the underlying $\omega$ is the critical percolation on the triangular lattice.

\begin{remark}\label{r.AT}
The Ashkin--Teller model on a finite graph $G=(V,E)$ is a spin system where at each vertex $i\in V$, we place a pair of Ising spins $(\sigma_i,\tau_i)\in\{\pm1\}^2$. Its Hamiltonian is given by
\begin{equation*}
H(\sigma,\tau)
=
-\sum_{ij\in E}
\Big(
J_\sigma\,\sigma_i\sigma_j
+
J_\tau\,\tau_i\tau_j
+
J_{\sigma\tau}\,\sigma_i\sigma_j\tau_i\tau_j
\Big),
\end{equation*}
where $J_\sigma,J_\tau,J_{\sigma\tau}$ are coupling constants.
At the Potts point $J_\sigma = J_\tau = J_{\sigma\tau} = \frac{1}{4}\log\frac{1}{1-p}$, the Ashkin--Teller model is equivalent to the $4$-Potts model by mapping $\{\pm1\}^2$ to the four colors.
Since we work with the critical model, we take $p=p_\mathrm{c}=2/3$.
Thus, the spin $\S$ corresponds exactly to projecting the Ashkin--Teller spin $(\sigma,\tau)$ onto one of its components (e.g., $\sigma$).

At the Potts point, the polarization field $\sigma\tau$ has the same law as either spin field $\sigma$ or $\tau$. Hence Theorems~\ref{t} and~\ref{t.Phi} also apply to the polarization field, with the same normalization.

Recent preprints~\cite[Conjecture~1.1]{alcalde2026ising} and~\cite[Conjecture~1.7]{alcalde2026excursion} formulate scaling-limit conjectures for the Ashkin--Teller polarization field along the critical line and its excursion decomposition. At the Potts endpoint, corresponding to imaginary-chaos parameter $1/2$, our results identify the full-plane polarization field in law, as explained below Theorem~\ref{t.Phi}.
\end{remark}

For every $z\in \R^2$ and $\delta>0$, let $z_\delta$ be the closest point on $\delta\Z^2$ (if there is more than one nearest point, we choose one according to a fixed order).
Throughout, we denote by $|\,\cdot\,|$ the Euclidean norm on $\R^2$ and we set $B_r(z) = \Ll\{z'\in\R^2:\:|z'-z|\leq r\Rr\}$ for $z\in\R^2$ and $r\geq0$.

If $n$ is odd, then we have $\E_\delta\Ll[ \S_{(x_1)_\delta} \cdots \S_{(x_n)_\delta} \Rr] = 0$, since there must exist a cluster $\mathcal{C}$ containing an odd number of the vertices $(x_1)_\delta, \dots, (x_n)_\delta$. In this case, the average of $\prod_{(x_i)_\delta \in \mathcal{C}} \S_{(x_i)_\delta} = \boldsymbol{\sigma}_{\mathcal{C}}$ vanishes. If $n$ is even, a similar argument shows that $\E_\delta\Ll[ \S_{(x_1)_\delta} \cdots \S_{(x_n)_\delta} \Rr]$ equals the probability that $(x_1)_\delta, \dots, (x_n)_\delta$ are partitioned into distinct clusters, each containing an even number of these vertices. This event is determined solely by the configuration $\omega$, which allows us to apply the relevant tools. 

The limit of the two-point function has been computed in~\cite[Proposition~1.3]{qequal4} and we have the following scaling:
\begin{align}\label{e.varrho=}
    \varrho(\delta):=\frac{1}{2}\E_\delta \Ll[\S_{(0,0)_\delta}\S_{(1,0)_{\delta}}\Rr]=\delta^{\frac{1}{4}+o(\delta)},\qquad\text{as $\delta\to0$.}
\end{align}
The error term $o(\delta)$ is not precise enough to capture the logarithmic correction predicted for the $4$-Potts model in the physics literature. The expected behavior is $\varrho(\delta)^\frac{1}{2}\asymp \delta^{\frac{1}{8}}(-\log \delta)^{-\frac{1}{16}}$, where the leading order was identified in~\cite{den1983extended} and the logarithmic correction in~\cite{cardy1980scaling,salas1997logarithmic}.
Our main result is as follows. 
Throughout, we take $\N=\{1,2,\dots\}$ excluding $0$.

\begin{theorem}\label{t}
For every $n\in 2\N$ and distinct points $(x_1,\cdots,x_n)\in(\R^2)^n$, we have
\begin{equation}\label{e.t}
    \lim_{\delta\rightarrow 0} \varrho(\delta)^{-\frac{n}{2}}\E_\delta\Ll[ \S_{(x_1)_\delta} \cdots \S_{(x_n)_\delta} \Rr] = \sum_{\substack{\xi\in\{-1,+1\}^n\\\sum_{i=1}^n\xi_i=0}}\prod_{1\leq i<j\leq n}|x_i-x_j|^{\frac{\xi_i\xi_j}{4}}.
\end{equation}
Moreover, the convergence is uniform over every compact subset of $\{(x_1,\dots,x_n)\in(\R^2)^n:\: x_i\neq x_j\ \text{for all $i\neq j$}\}$.
\end{theorem}

The right-hand side of \eqref{e.t} defines a conformally covariant function corresponding to the correlation of fields with scaling dimension equal to $1/8$. For $n=2$, it is equal to $2|x_1-x_2|^{-1/4}$ (already proven in~\cite[Proposition~1.3]{qequal4}), and for $n=4$, it is equal to
\[
2\frac{|x_1-x_2|^{1/2}|x_3-x_4|^{1/2} + |x_1-x_3|^{1/2}|x_2-x_4|^{1/2} + |x_1-x_4|^{1/2}|x_2-x_3|^{1/2}}{|x_1-x_2|^{1/4}|x_1-x_3|^{1/4}|x_1-x_4|^{1/4}|x_2-x_3|^{1/4}|x_2-x_4|^{1/4}|x_3-x_4|^{1/4}}.
\]

It is then natural to consider the rescaled spin field 
\begin{align}\label{e.Phi_delta=}
   \Phi_\delta := \delta^2 \varrho(\delta)^{-\frac{1}{2}}\sum_{z\in\delta\Z^2} \S_z \boldsymbol{\delta}_z,  
\end{align}
where $\boldsymbol{\delta}_z$ denotes the Dirac mass at $z$. 
We view $\Phi_\delta$ as a random distribution on $\R^2$.
For $s>0$, we denote by $\mcl H^{-s}(\R^2)$ and $\mcl H^s_0(\R^2)$ the
$L^2$-based Sobolev spaces described in Section~\ref{s.sobolev}. For every
$n\in\N$ and every $h\in\mcl H^s_0(\R^2)$ for some $s>1$, we define
\begin{align}
    M_n(h):=
    \begin{cases}
    \displaystyle \int_{(\R^2)^n}
    h(x_1)\cdots h(x_n)\la x_1,\ldots,x_n\ra
    \,dx_1\cdots dx_n, & n\in2\N,\\
    0, & n\notin2\N,
    \end{cases}
    \label{e.M_n(h)=}
\end{align}
where $\la x_1,\ldots,x_n\ra$ denotes the right-hand side of~\eqref{e.t}.
Since $s>1$, such $h$ is bounded and compactly supported. Moreover, the
singularities of $\la x_1,\ldots,x_n\ra$ along the diagonals are locally
integrable, and hence $M_n(h)$ is well-defined.

\begin{theorem}\label{t.Phi}
For any $s>1$, the magnetization field $\Phi_\delta$ converges in law to a
random distribution $\Phi_0$ in $\mathcal H^{-s}(\mathbb R^2)$ as $\delta\to0$.
Moreover, the law of $\Phi_0$ is uniquely characterized by the
moment identities $\mathbb E\left[\left(\int \Phi_0 h\right)^n\right] = M_n(h)$ for every $h\in\mathcal H^s_0(\mathbb R^2)$
and every $n\in\N$.
In particular, for each such $h$, the law of $\int\Phi_0h$ is determined by $(M_n(h))_{n\in\N}$.
\end{theorem}

By Theorem~\ref{t}, we can show that $\lim_{\delta \to 0} \E_\delta\Ll[\Ll(\int \Phi_\delta h\Rr)^n\Rr] = M_n(h)$
by adapting the arguments in~\cite{camia2015planar} for the magnetization field of the Ising model. In~\cite{camia2015planar}, the limiting magnetization field is identified through the convergence of characteristic functions, after establishing uniform bounds on the exponential moments of the magnetization field using the GHS inequality. In the present setting, however, this inequality is not available. Consequently, we cannot obtain the corresponding uniform exponential-moment bounds, and we instead identify the limit through its moments by invoking Carleman's condition.

Full-plane imaginary Gaussian multiplicative chaos can be constructed rigorously; see~\cite[Section~6.1]{newman2019leeyang} and~\cite[the discussion preceding Theorem~1.6]{aru2025noiselike}. Consequently, the limiting field in Theorem~\ref{t.Phi} can be identified in law with $2\mathopen{:}\cos(X/2)\mathclose{:}$ for a full-plane GFF $X$ with the appropriate global phase.

Throughout, we use the following notation. If $g$ and $g'$ are positive functions on a common domain $\mathcal X$, then $g \lesssim g'$ means that there exists a constant $C>0$ such that $g(x) \le C g'(x)$ for all $x \in \mathcal X$.
We write $g \asymp g'$ when both $g \lesssim g'$ and $g' \lesssim g$ hold. The domain $\mathcal X$ will be left implicit whenever it is clear from context.

\section{Proof of Theorem~\ref{t} modulo some results}

It suffices to prove the following.

\begin{proposition}\label{p}
There is a factor $f(\delta)=\delta^{\frac{1}{8}+o(1)}$ such that, for every $n\in 2\N$ and distinct points $(x_1,\cdots,x_n)\in(\R^2)^n$, we have
\begin{equation}\label{e.p}
    \lim_{\delta\rightarrow 0} f(\delta)^{-n}\E_\delta\Ll[ \S_{(x_1)_\delta} \cdots \S_{(x_n)_\delta} \Rr] = 2^{-n}e^{c_0n} \sum_{\substack{\xi\in\{-1,+1\}^n\\\sum_{i=1}^n\xi_i=0}}\prod_{1\leq i<j\leq n}|x_i-x_j|^{\frac{\xi_i\xi_j}{4}}
\end{equation}
where $c_0  = \frac{1}{8\pi^2}\iint_{B_1(0)\times B_1(0)}\log|\rz-\rz'|\d \rz\d \rz'$.\footnote{In fact, one can compute explicitly that $c_0=-\frac{1}{32}$, but this value is inconsequential.} Moreover, the convergence is uniform over every compact subset of $\{(x_1,\dots,x_n):\: x_i\neq x_j\ \text{for all $i\neq j$}\}$.
\end{proposition}

\begin{proof}[Proof of Theorem~\ref{t}]
Recall $\varrho(\delta)$ from~\eqref{e.varrho=}. Then, using~\eqref{e.p} with $n=2$, $x_1=(0,0)$, and $x_2=(1,0)$, we get $\lim_{\delta\to0}f(\delta)^{-2}(2\varrho(\delta))=2^{-1}e^{2c_0}$ and thus $\lim_{\delta\to0}\frac{2e^{-c_0}\varrho(\delta)^\frac{1}{2}}{f(\delta)}=1$. From this relation, we can obtain~\eqref{e.t} from~\eqref{e.p}.
\end{proof}

Henceforth, we focus on proving Proposition~\ref{p}.
One important ingredient of the proof is Theorem~\ref{thm:M formula}, which combines the Baxter--Kelland--Wu coupling \cite{baxter1976equivalence} with the convergence of the six-vertex model to the Gaussian free field~\cite{GFF6vertex}.
A collection \(L\) of non-self-crossing loops is said to be \textit{locally finite} if, for every compact \(D\subset \R^2\), there are only finitely many loops in \(L\) that intersect \(D\) nontrivially. For every integrable, compactly supported function \(F:\R^2\to\R\) satisfying \(\int_{\R^2}F(z)\,\d z\in 2\pi\Z\), and every locally finite \(L\), we define
\begin{equation}\label{e.A_F(L)=}
\mathcal A_F(L):=\prod_{\ell\in L}\cos\Ll(\int_{\itr(\ell)} F(z)\,\d z\Rr)
\end{equation}
where $\itr(\lo)$ is the complement of the unbounded component of $\R^2\setminus\ell$.
By local finiteness, only finitely many loops can separate different parts of
$\supp F$; all sufficiently large loops surrounding $\supp F$ contribute
$\cos(\int_{\R^2}F)=1$. Thus $\mathcal A_F(L)$ is well-defined.
Throughout, given a percolation configuration $\omega$, we denote by
\begin{align}\label{e.cL=}
    \lpcfg = \lpcfg(\omega) :=\Ll\{\text{interface loops of $\omega$}\Rr\}
\end{align}
the associated loop configuration of $\omega$.

\begin{theorem}[{\cite[Theorem 3.1]{qequal4}}]\label{thm:M formula}
    For every mean-zero $F\in \mcl H^s_0(\R^2)$ for some~$s\in(0,1)$, we have
    \begin{equation*}\lim_{\delta\rightarrow 0} \rcE_{\delta \Z^2,p_\crit,4}\Ll[\mathcal A_F(\lpcfg)\Rr]=\exp\Ll( \frac{1}{2\pi^2} \iint_{\R^2\times\R^2} F(z)F(z')\log|z-z'| \d z\d z' \Rr).
    \end{equation*}
\end{theorem}

We refer to Section~\ref{s.sobolev} for details of the Sobolev space $\mcl H^s_0(\R^2)$.
In the proof, we also need to consider one-arm events as follows.
For $0\leq r<R<\infty$, we define
\begin{align*}\pi_1(r,R):= \phi_{\Z^2}\Ll[[-r,r]^2\longleftrightarrow (\R^2\setminus(-R,R)^2)\Rr]\qquad\text{and}\qquad\pi_1(R):=\pi_1(0,R).
\end{align*}

We introduce some notation needed for the proof.
Fix $n\in 2\N$ and fix points $x_1,x_2,\dots,x_n$ in $\R^2$. 
Let $\pa = (P_j,p_j)_{1\leq j\leq J}$ denote a connection pattern on percolation configurations which is specified by an integer $J\in\N$, a partition $(P_j)_{1\leq j\leq J}$ of the index set $\{1,2,\dots,n\}$ satisfying that each $|P_j|$ is even, and a sequence of numbers $(p_j)_{1\leq j\leq J}$ where $p_j=0$ (resp.\ $p_j=1$) indicates $\{x_i\}_{i\in P_j}$ are dual-connected (resp.\ primal-connected). We denote by $\mathfrak{P}$ the collection of such $\pa$.
We also need the subcollections
\begin{align}\label{e.frakP_+=}
    \mathfrak{P}_+:=\Ll\{\pa\in \mathfrak{P}:\:\text{$(p_j)$ of $\pa$ are all equal}\Rr\}\quad\text{and}\quad \mathfrak{P}_{+,s}:=\Ll\{\pa\in \mathfrak{P}_+:\:p_j=s,\,\forall j\Rr\} \text{ for $s\in\{0,1\}$}.
\end{align}
We clearly have $\mathfrak{P}_+=\mathfrak{P}_{+,1}\sqcup \mathfrak{P}_{+,0}$. 

For $s\in\{0,1\}$, an $s$-connection means a primal connection if $s=1$
and a dual connection if $s=0$. 
For each $\eps>0$, we consider the event $E_\eps(\pa)$ for $\omega$ that
\begin{itemize}
    \item for every $j$, $(B_\eps(x_i))_{i\in P_j}$ are mutually $p_j$-connected in $\omega$;
    \item for distinct $j$ and $j'$, there is no primal or dual connection in $\omega$ between $(B_\eps(x_i))_{i\in P_j}$ and $(B_\eps(x_i))_{i\in P_{j'}}$.
\end{itemize}
When $\eps=0$, we define $E_0(\pa)$ similarly, but with $p_j$-connections understood as among $\Ll\{(x_i)_\delta\Rr\}_{i\in P_j}$ if $p_j=1$ and among $\Ll\{(x_i)_\delta+\frac{\delta}{2}\Ll(1,1\Rr)\Rr\}_{i\in P_j}$ if $p_j=0$.

Notice that $\mathfrak{P}_{+,1}$ describes an even partition of the points $x_1,\dots,x_n$ into open clusters. The definition of $\S$ in~\eqref{e.S=} together with discussion above~\eqref{e.varrho=} implies the first equality in
\begin{align}\label{e.<S..S>=sum_P}
    \E_\delta\Ll[ \S_{(x_1)_\delta}\cdots \S_{(x_n)_\delta}\Rr] = \sum_{\pa\in\mathfrak{P}_{+,1}}\phi_{\delta\Z^2}\Ll[E_0(\pa)\Rr] = \frac{1}{2}\sum_{\pa\in\mathfrak{P}_+}\phi_{\delta\Z^2}\Ll[E_0(\pa)\Rr], 
\end{align}
where the second equality follows from $\sum_{\pa\in\mathfrak{P}_{+,1}}\phi_{\delta\Z^2}\Ll[E_0(\pa)\Rr]=\sum_{\pa\in\mathfrak{P}_{+,0}}\phi_{\delta\Z^2}\Ll[E_0(\pa)\Rr]$.

Let $\xi=(\xi_1,\cdots ,\xi_n)\in\{\pm1\}^n$ denote the assignment of weights $\xi_i\pi/2$ on $\eps$-balls centered at $x_i$ for every $i\in\llbracket1,n\rrbracket$. 
For $\eps>0$ and $\xi$, writing $x=(x_i)_{i\in\llbracket1,n\rrbracket}$, we consider the test function
\begin{align}\label{e.F_xi,eps=}
    F_{x,\xi,\eps} := \sum_{i=1}^n \frac{\xi_i}{2\eps^2}\one_{B_\eps(x_i)}.
\end{align}
For $\tau\in 2\Z$, we consider the collection
\begin{align}\label{e.Xi_tau=}
    \Xi_\tau :=\Ll\{\xi\in\{\pm1\}^n:\:\sum_{i=1}^n\xi_i =\tau\Rr\}.
\end{align}
Since \(B_\eps(x_i)\) has Lipschitz boundary, $F_{x,\xi,\eps}$ belongs to $\mcl H^s_0(\R^2)$ for every $s<\frac{1}{2}$ (see~\cite[Lemma~4]{sickel2020regularity}). Hence, whenever $\tau=0$, Theorem~\ref{thm:M formula} applies.

Throughout, we write $\llbracket r,s\rrbracket=[r,s]\cap\Z$. Given $\xi\in\Xi_\tau$ for some $\tau$ and $\pa\in \mathfrak{P}$, we define
\begin{align}
    b_j  = b_j(\xi,\pa)  := \frac{1}{2}\sum_{i\in P_j}\xi_i,\quad\forall j \in \llbracket1,J\rrbracket \qquad \text{and} \qquad a(\xi,\pa) := (-1)^{\sum_{j=1}^Jb_jp_j}.  \label{e.b_j=a(xi,pi)=}
\end{align}

\begin{proof}[Proof of Proposition~\ref{p} modulo~\eqref{e.monday}, \eqref{e.compute_a=}, \eqref{e.GFF_compute}, and~\eqref{e.f_n,d_ratio}]

Fix any $n\in 2\N$ and any $\dst \in (0,1)$. Set
\begin{align}\label{e.X(n,d)=}
    X(n,\dst):=\Ll\{x\in(\R^2)^n:\:\max_i|x_i|\leq \frac{1}{d},\, \min_{i\neq j}|x_i-x_j|\geq d\Rr\}.
\end{align}
In Lemma~\ref{l.e.monday}, we will expand $\mathcal{A}_{F_{x,\xi,\eps}}(\mcl L)$ over connection patterns and get
\begin{align}\label{e.monday}
    \phi_{\delta\Z^2}\Ll[\mathcal{A}_{F_{x,\xi,\eps}}(\mcl L)\Rr]=\bc_\delta(\eps)^n\sum_{\pa\in\mathfrak{P}}a(\xi,\pa)\phi_{\delta\Z^2}\Ll[E_\eps(\pa)\Rr]  +O\Ll(\eps^{\frac{n}{8}+c+o_{\delta/\eps}(1)}\Rr), \tag{A}
\end{align}
uniformly in $x\in X(n,\dst)$, for each $\xi\in\Xi_\tau$ with $\tau\in4\Z$, where $c>0$ is a constant, $\bc_\delta(\eps)$ is a deterministic factor given in~\eqref{e.bc^delta=} which, by Lemma~\ref{l.kappa}, is bounded from above and below away from zero as long as $\delta/\eps$ is small. Let us clarify the error term. 
There is some constant $C$ such that $|O\Ll(\eps^{\frac{n}{8}+c+o_{\delta/\eps}(1)}\Rr)|\leq C \eps^{\frac{n}{8}+c+o_{\delta/\eps}(1)}$ uniformly in $x\in X(n,\dst)$, for every $0<\delta<\eps<\dst/2$, where $o_{\delta/\eps}(1)$ only depends on $\delta/\eps$ and vanishes as $\delta/\eps$ tends to $0$.

A key combinatorial fact recorded in Lemma~\ref{l.cancel} is that
\begin{align}\label{e.compute_a=}
\sum_{\tau\in4\Z\cap[-n,n]}\sum_{\xi\in\Xi_\tau} a(\xi,\pa)=
\begin{cases}
    2^{n-1},\qquad&\forall \pa \in \mathfrak{P}_+,  
    \\
    0,\qquad&\forall \pa\notin\mathfrak{P}_+.
\end{cases} \tag{B}
\end{align}
Now, summing~\eqref{e.monday} over $\xi\in\Xi_\tau$ and $\tau\in 4\Z\cap[-n,n]$ and using~\eqref{e.compute_a=}, we get
\begin{align}\label{e.sum_P=sum_phi}
    \bc_\delta(\eps)^n2^{n-1}\sum_{\pa\in\mathfrak{P}_+}\phi_{\delta\Z^2}\Ll[E_\eps(\pa)\Rr] = \sum_{\tau\in4\Z\cap[-n,n]}\sum_{\xi\in\Xi_\tau} \phi_{\delta\Z^2}\Ll[\mathcal{A}_{F_{x,\xi,\eps}}(\mcl L)\Rr] +O\Ll(\eps^{\frac{n}{8}+c+o_{\delta/\eps}(1)}\Rr).
\end{align}

The left-hand side of~\eqref{e.sum_P=sum_phi} can be related to~\eqref{e.<S..S>=sum_P} via
\begin{align}\label{e.E_0(P)-E_eps(P)}
    \Ll|\phi_{\delta\Z^2}\Ll[E_0(\pa)\Rr] - g_\eps(\delta)^n \phi_{\delta\Z^2}\Ll[E_\eps(\pa)\Rr]\Rr|\leq C'\eps^{c'} \phi_{\delta\Z^2}\Ll[E_0(\pa)\Rr]
\end{align}
uniformly over $(x_1,\dots,x_n)\in(\R^2)^n$ satisfying $\max_i|x_i|\leq 1/\dst$, for some constants $C',c'>0$,
where
\begin{align}\label{e.g_eps(delta)_control}
    g_\eps(\delta):= \phi_{\delta\Z^2}[0\longleftrightarrow \infty\,|\,B_\eps(0)\longleftrightarrow\infty] \asymp \pi_1(\eps \delta^{-1})=(\eps\delta^{-1})^{-1/8+o_{\delta/\eps}(1)}\qquad\text{as }\delta/\eps\to0. 
\end{align}
The estimate~\eqref{e.E_0(P)-E_eps(P)} follows from the mixing property conditioned on one-arms, which can be adapted from the coupling argument in~\cite[Section~3]{garban2013pivotal}\footnote{A similar form of~\eqref{e.E_0(P)-E_eps(P)} was used in \cite[Proof of Corollary~1.4]{duminil2020rotational} that also adapts the said argument from~\cite{garban2013pivotal}. Also, see~\cite[(4.9)]{qequal4} for a similar version used to compute the two-point function.}.
The bound on $g_\eps(\delta)$ in~\eqref{e.g_eps(delta)_control} follows from the one-arm exponent rigorously computed in~\cite[Theorem~1.2]{qequal4}.

The right-hand side in~\eqref{e.sum_P=sum_phi} will be computed explicitly in Lemmas~\ref{l.GFF_test_function} and~\ref{l.e.GFF_compute} using Theorem~\ref{thm:M formula}:
\begin{equation}\label{e.GFF_compute}
\begin{alignedat}{2}
\lim_{\delta\to0}\eps^{-n/8}\phi_{\delta\Z^2}\Ll[\mathcal{A}_{F_{x,\xi,\eps}}(\mcl L)\Rr]
&= e^{nc_0}(1+o_\eps(1))\prod_{1\le i<j\le n}|x_i-x_j|^{\frac{\xi_i\xi_j}{4}},
\qquad &\forall \xi\in\Xi_0,
\\
\lim_{\delta\to0}\eps^{-n/8}\phi_{\delta\Z^2}\Ll[\mathcal{A}_{F_{x,\xi,\eps}}(\mcl L)\Rr]
&= 0, \qquad &\forall \xi\in\Xi_\tau,\ \forall \tau\in 4\Z\setminus\{0\},
\end{alignedat} \tag{C}
\end{equation}
for some absolute constant $c_0$, where the convergence is uniform in $x\in X(n,\dst)$.

We are ready to conclude. Let $\rho(t)\to0$ (as $t\to0$) dominate all occurrences of
$o_{\delta/\eps}(1)$ in the estimates above. We may
choose $\eps(\delta)>0$ decaying slowly (see Lemma~\ref{lem:choose_epsilon} for the detail) so that
\begin{align}
    \eps(\delta)\to0,\qquad \tfrac{\delta}{\eps(\delta)}\to0,\qquad
    \eps(\delta)^{o_{\delta/\eps(\delta)}(1)}\to1, \label{e.eps(delta)_properties}
\end{align}
and such that \eqref{e.GFF_compute} holds with $\eps$ replaced by $\eps(\delta)$, uniformly in
$x\in X(n,d)$. Notice that~\eqref{e.monday} and thus~\eqref{e.sum_P=sum_phi} also hold with
$\eps=\eps(\delta)$.

Inserting~\eqref{e.E_0(P)-E_eps(P)} and~\eqref{e.GFF_compute} (with $\eps(\delta)$) into~\eqref{e.sum_P=sum_phi} and recalling~\eqref{e.<S..S>=sum_P}, we get
\begin{align}\label{e.cvg_f_n,d}
    \lim_{\delta\to0}f_{n,\dst}(\delta)^{-n} \E_\delta\Ll[ \S_{(x_1)_\delta}\cdots \S_{(x_{n})_\delta}\Rr] = 2^{-n}e^{nc_0}\sum_{\xi\in\Xi_0}\prod_{1\leq i<j\leq n}|x_i-x_j|^{\frac{\xi_i\xi_j}{4}}
\end{align}
uniformly in $x\in X(n,\dst)$,
where we have set $f_{n,\dst}(\delta) = \frac{\eps(\delta)^{1/8}g_{\eps(\delta)}(\delta)}{\bc_\delta(\eps(\delta))}$. 
Using~\eqref{e.g_eps(delta)_control}, $\bc_\delta(\eps)\asymp 1$ for $\delta/\eps$ sufficiently small, and~\eqref{e.eps(delta)_properties}, we have $f_{n,\dst}(\delta)=\delta^{1/8+o_\delta(1)}$.
Lastly, we will show in Lemma~\ref{l.f_n,d} that
\begin{align}\label{e.f_n,d_ratio}
    \lim_{\delta\to0} \frac{f_{n,\dst}(\delta)}{f_{2,\frac{1}{2}}(\delta)}=1 \qquad\text{for every $n\in 2\N$ and $d\in(0,1)$}. \tag{D}
\end{align}
Hence, by setting $f(\delta)=f_{2,\frac{1}{2}}(\delta)$, the desired result in~\eqref{e.p}, together with the local uniformity of the convergence, follows from~\eqref{e.cvg_f_n,d} and~\eqref{e.f_n,d_ratio}.
This completes the proof of Proposition~\ref{p} modulo \eqref{e.monday}, \eqref{e.compute_a=}, \eqref{e.GFF_compute}, and~\eqref{e.f_n,d_ratio}, which will be proved in Section~\ref{s.rem_proofs}.
\end{proof}

\section{Remainder of the proof of Proposition~\ref{p}}\label{s.rem_proofs}

\subsection{Derivation of~(\ref{e.monday}): loop observables and connection patterns}

Write $x=(x_i)_{i\in\llbracket1,n\rrbracket}$ and, slightly abusing the notation, we also view $x$ as the set of these points.
Recall $\mcl L$ from~\eqref{e.cL=}.
Given $x$ and $\eps>0$, we set
\begin{align*}\calL_{x,\eps} &= \Ll\{\text{loops in $\mathcal L$ that do not intersect any of the balls $B_\eps(x_i)$ for $1\le i\le n$}\Rr\} . \end{align*}
For every $0<\delta<\eps$, we consider the {\em normalization factor}
\begin{equation}\label{e.bc^delta=}
    \bc_\delta(\eps):=\phi_{\delta\Z^2}\Ll[\mathcal A_{\frac1{2\eps^2}\one_{B_\eps(0)}}(\mathcal L)\ \Big|\ 0\notin\itr(\ell),\,\forall \ell \in \calL\Rr].
\end{equation}
Although the test function in~\eqref{e.bc^delta=} does not have total mass in
$2\pi\Z$, under the conditioning no loop surrounds $B_\eps(0)$, and the product
is interpreted in the same finite sense.

We need the following inputs from~\cite{qequal4}.

\begin{lemma}[{\cite[Lemma~3.3]{qequal4}}]\label{l.kappa}

There exists $c>0$ such that $c\le \bc_\delta(\eps)\le 1$ for every $0<\delta<c\eps$.
\end{lemma}

\begin{proposition}[{\cite[Proposition~3.5]{qequal4}}]\label{prop:forget small loops}
There exists $c>0$ such that the following holds. For every $n\in\N$ and $\dst>0$, there is a constant $C>0$ such that for every $0<\delta\le \eps<\dst/2$, every set of $n$ points $x\in X(n,\dst)$ (see~\eqref{e.X(n,d)=}), and every $\xi\in\Xi_\tau$ (see~\eqref{e.Xi_tau=}) for some $\tau\in4\Z$, we have
\begin{equation*}\Big|\bc_\delta(\eps)^n\phi_{\delta\Z^2}\Ll[\mathcal A_{F_{x,\xi,\eps}}\Ll(\mathcal L_{x,\eps}\Rr)\Rr]-\phi_{\delta\Z^2}[\mathcal A_{F_{x,\xi,\eps}}(\mathcal L)]\Big|\le C\pi_1\Ll(\tfrac{\eps}{\delta},\tfrac{\dst}{\delta}\Rr)^n(\tfrac\eps\dst)^c.\end{equation*}

\end{proposition}

The relation in~\eqref{e.monday}, together with the uniformity condition stated below the display, follows from the lemma below.
We need two simple facts derived from~\eqref{e.b_j=a(xi,pi)=} and the definition of $\mathfrak{P}_+$ in~\eqref{e.frakP_+=}:
\begin{align}
    \tau\in 4\Z,\quad  \xi \in \Xi_\tau\qquad &\Longrightarrow\qquad \sum_{j=1}^J b_j =\tau/2\in2\Z\label{e.sumb_j=0mod2} ;
    \\
    \label{e.p_jequal=>a=1}
    \tau\in 4\Z,\quad \xi \in \Xi_\tau,\quad \pa\in\mathfrak{P}_+ \qquad &\Longrightarrow \qquad a(\xi,\pa)=1.
\end{align}

\begin{lemma}\label{l.e.monday}
There exists $c>0$ such that the following holds. For every $n\in 2\N$ and $\dst>0$, there is a constant $C>0$ such that for every $0<\delta\le \eps<\dst/2$, every set of $n$ points $x\in X(n,\dst)$, and every $\xi\in\Xi_\tau$ for some $\tau\in4\Z$, we have
\begin{equation*}
\Ll|\bc_\delta(\eps)^n \sum_{\pa\in\mathfrak{P}}a(\xi,\pa)\phi_{\delta\Z^2}\Ll[E_\eps(\pa)\Rr]-\phi_{\delta\Z^2}[\mathcal A_{F_{x,\xi,\eps}}(\mathcal L)]\Rr|\le C\pi_1(\tfrac{\eps}{\delta},\tfrac{\dst}{\delta})^n(\tfrac\eps\dst)^c.\end{equation*}
\end{lemma}

\begin{proof}
Fix $\pa\in\mathfrak{P}$. For $\ell\in\mcl L_{x,\eps}$, set
\begin{align*}
    I_\ell:=\Ll\{i\in\llbracket1,n\rrbracket:B_\eps(x_i)\subset\itr(\ell)\Rr\}
    \qquad\text{and}\qquad
    J_\ell:=\Ll\{j\in\llbracket1,J\rrbracket:P_j\subset I_\ell\Rr\}.
\end{align*}
On $E_\eps(\pa)$, each $I_\ell$ is a union of blocks $P_j$, since an interface loop cannot be crossed by a primal or dual connection.  Define
\begin{align*}
    N_j:=\#\Ll\{\ell\in\mcl L_{x,\eps}: j\in J_\ell,
    \ J_\ell\neq \llbracket1,J\rrbracket\Rr\},\qquad \forall j\in\llbracket1,J\rrbracket,
\end{align*}
so loops surrounding all of the $\eps$-balls are not counted. Due to
\begin{align*}
    \int_{\itr(\ell)}F_{x,\xi,\eps}=\frac{\pi}{2}\sum_{i\in I_\ell}\xi_i
    \stackrel{\eqref{e.b_j=a(xi,pi)=}}{=}\pi\sum_{j\in J_\ell}b_j,
\end{align*}
and~\eqref{e.sumb_j=0mod2}, loops with $J_\ell=\varnothing$ or $J_\ell=\llbracket1,J\rrbracket$ contribute $1$ to $\mathcal A_{F_{x,\xi,\eps}}(\mcl L_{x,\eps})$. Hence, we have
\begin{align*}
    \mathcal A_{F_{x,\xi,\eps}}(\mcl L_{x,\eps})
    =\prod_{\substack{\ell\in\mcl L_{x,\eps}:\; J_\ell\neq\varnothing,\\
    J_\ell\neq\llbracket1,J\rrbracket}}
    (-1)^{\sum_{j\in J_\ell}b_j}
    =(-1)^{\sum_{j=1}^J b_jN_j},\qquad \text{on $E_\eps(\pa)$}.
\end{align*}
Moreover, crossing an interface loop changes the primal/dual type, and therefore
\begin{align*}
    N_j-N_{j'}=p_j-p_{j'}\pmod{2},
    \qquad \forall j,j'\in\llbracket1,J\rrbracket.
\end{align*}
It follows that
\begin{align*}
    \sum_{j=1}^Jb_jN_j
    =\sum_{j=1}^Jb_j(p_j-p_1+N_1)
    \stackrel{\eqref{e.sumb_j=0mod2}}{=}
    \sum_{j=1}^Jb_jp_j\pmod{2},
\end{align*}
and thus (recall $a(\xi,\pa)$ from~\eqref{e.b_j=a(xi,pi)=})
\begin{align}\label{e.AL=a}
    \mathcal A_{F_{x,\xi,\eps}}(\mcl L_{x,\eps})
    =a(\xi,\pa),\qquad \text{on $E_\eps(\pa)$}.
\end{align}

On the other hand, let $E_\mathrm{touch}$ be the event that there is a loop $\lo$ satisfying $\lo\cap B_\eps(x_i)\neq \emptyset$ and $\itr(\lo)\cap B_\eps(x_j)\neq \emptyset$ for some $i\neq j$.
If $\omega\notin\Ll(\bigsqcup_{\pa\in\mathfrak{P}}E_\eps(\pa)\Rr)\cup E_\mathrm{touch}$, then, by the same planar loop/cluster correspondence, some loop $\ell\in\mcl L_{x,\eps}$ surrounds an odd number of the balls $B_\eps(x_i)$.
For such a loop, we have $\cos\Ll(\int_{\itr(\ell)}F_{x,\xi,\eps}\Rr)=0$ and thus $\mathcal A_{F_{x,\xi,\eps}}(\mcl L_{x,\eps})=0$ on the complement of $\Ll(\bigsqcup_{\pa\in\mathfrak{P}}E_\eps(\pa)\Rr)\cup E_\mathrm{touch}$. Therefore, we have
\begin{align}
    \Ll|\phi_{\delta\Z^2}\Ll[\mathcal A_{F_{x,\xi,\eps}}(\mcl L_{x,\eps})\Rr] - \sum_{\pa\in\mathfrak{P}}\phi_{\delta\Z^2}\Ll[\mathcal A_{F_{x,\xi,\eps}}(\mcl L_{x,\eps})\one_{E_\eps(\pa)}\Rr] \Rr| &\leq \phi_{\delta\Z^2}\Ll[\Ll|\mathcal A_{F_{x,\xi,\eps}}(\mcl L_{x,\eps})\Rr|\one_{E_\mathrm{touch}}\Rr] \notag
    \\
    & \leq \phi_{\delta\Z^2}\Ll[E_\mathrm{touch}\cap\Ll\{\mathcal A_{F_{x,\xi,\eps}}(\mcl L_{x,\eps})\neq 0\Rr\}\Rr] \label{e.|A-A|<}.
\end{align}
In the last inequality, we used the fact $|\mathcal A_{F_{x,\xi,\eps}}(\mcl L_{x,\eps})|\leq 1$, which is evident from~\eqref{e.A_F(L)=}.
On $E_\mathrm{touch}$, since there must be a loop touching some $B_\eps(x_i)$ and having diameter larger than $\dst/4$, there is a two-arm event from scale $\eps/\delta$ to $\dst/4\delta$ centered at $x_i$. To ensure $\mathcal A_{F_{x,\xi,\eps}}(\mcl L_{x,\eps})\neq 0$, for each $j\neq i$, there cannot be a loop surrounding $B_\eps(x_j)$ inside $B_{\dst/4}(x_j)$; hence, there must be a one-arm event from scale $\eps/\delta$ to $\dst/4\delta$ centered at $x_j$. These arm events are well-separated, which allows us to use the mixing property (see~\cite[Proposition~2.9]{duminil2022planar}) to deduce that
\begin{align*}
    \phi_{\delta\Z^2}\Ll[E_\mathrm{touch}\cap\Ll\{\mathcal A_{F_{x,\xi,\eps}}(\mcl L_{x,\eps})\neq 0\Rr\}\Rr] \lesssim \pi_1(\tfrac{\eps}{\delta},\tfrac{\dst}{\delta})^{n-1}\pi_2(\tfrac{\eps}{\delta},\tfrac{\dst}{\delta})\lesssim \pi_1(\tfrac{\eps}{\delta},\tfrac{\dst}{\delta})^n(\tfrac\eps\dst)^c
\end{align*}
where $\pi_2$ denotes the two-arm probability and $c>0$ is an absolute constant.
This along with~\eqref{e.|A-A|<} and~\eqref{e.AL=a} implies
\begin{align*}
    \Ll|\phi_{\delta\Z^2}\Ll[\mathcal A_{F_{x,\xi,\eps}}(\mcl L_{x,\eps})\Rr]
    -\sum_{\pa\in\mathfrak{P}}a(\xi,\pa)\phi_{\delta\Z^2}\Ll[E_\eps(\pa)\Rr]\Rr|\lesssim \pi_1(\tfrac{\eps}{\delta},\tfrac{\dst}{\delta})^n(\tfrac\eps\dst)^c.
\end{align*}
Inserting this to Proposition~\ref{prop:forget small loops} gives the desired result.  
\end{proof}

\subsection{Derivation of (\ref{e.compute_a=}): charge completion and combinatorial cancellation}

If $\pa\in\mathfrak{P}_+$, by \eqref{e.p_jequal=>a=1} and $|\Xi_\tau|=\binom{n}{\frac{n-\tau}{2}}$, the left-hand side of \eqref{e.compute_a=} becomes $ \sum_{\tau\in 4\Z\cap[-n,n]} \binom{n}{\frac{n-\tau}{2}}$, which is equal to $\sum_{m\in 2\Z\cap [0,n]}\binom{n}{m}= 2^{n-1}$ if $n=0\pmod 4$ and equal to $\sum_{m\in (2\Z+1)\cap [0,n]}\binom{n}{m}= 2^{n-1}$ if $n=2\pmod 4$. The other case is proved in the lemma below.

\begin{lemma}[Charge completion and combinatorial cancellation]\label{l.cancel}
Let $n\in 2\N$. If $\pa\notin\mathfrak{P}_+$ (see~\eqref{e.frakP_+=}), then
\begin{align}\label{e.sumsuma=0}
    \sum_{\tau\in 4\Z\cap[-n,n]}\ \sum_{\xi\in\Xi_\tau}a(\xi,\pa)=0.
\end{align}
\end{lemma}

\begin{proof}
Fix any $\pa$ as given by a partition $(P_j)_{1\leq j\leq J}$ and types of connections $(p_j)_{1\leq j\leq J}$ for some $J\in\N$.
We start with some notation. We define
\begin{align}\label{e.J_1=,J_0=}
    \mathfrak{J}_1 := \Ll\{j:\:p_j=1\Rr\}\qquad\text{and}\qquad \mathfrak{J}_0 := \Ll\{j:\:p_j=0\Rr\},
\end{align}
which contain the indices of partitions of points that are primal- and dual-connected, respectively. 
Since $p_j$'s are not all equal, both $\mathfrak{J}_1$ and $\mathfrak{J}_0$ are nonempty.
Given this $\pa$, for each $\tau\in 2\Z$ with $-n\leq\tau\leq n$, we define
\begin{align}\label{e.X(tau)=}
    X(\tau) := \Ll\{(\chi_j)_{j\in \llbracket 1,J\rrbracket}:\quad \chi_j \in \llbracket0,|P_j|\rrbracket,\,\forall j;\quad \sum_{j=1}^J\chi_j = \frac{n-\tau}{2}\Rr\}.
\end{align}
The interpretation of $X(\tau)$ is that, for each $j$, $\chi_j$ is a possible number of $-1$'s that $(\xi_i)_{i\in P_j}$ can take for $\xi\in\Xi_\tau$. 
A useful observation from~\eqref{e.X(tau)=} and $\sum_{j=1}^J|P_j|=n$ is that
\begin{align}\label{e.chi=bigcupX(tau)}
    \Ll\{(\chi_j)_{j\in \llbracket 1,J\rrbracket}:\quad \chi_j \in \llbracket0,|P_j|\rrbracket,\,\forall j\Rr\} = \bigcup_{\tau\in 2\Z\cap[-n,n]}X(\tau).
\end{align}
Accordingly, given each $\chi\in X(\tau)$, we define
\begin{align*}
    \Xi_\tau(\chi):= \Ll\{\xi\in \Xi_\tau:\quad\sum_{i\in P_j}\one_{\{\xi_i=-1\}}=\chi_j,\, \forall j\Rr\},
\end{align*}
which is the subcollection of $\Xi_\tau$ consisting of those $\xi$ such that $(\xi_i)_{i\in P_j}$ takes exactly $\chi_j$-many $-1$'s for each $j$.
So, we clearly have
\begin{align}\label{e.Xi_tau(chi)=...}
    \Xi_\tau = \bigsqcup_{\chi\in X(\tau)}\Xi_\tau(\chi).
\end{align}
Recall the definition of $b_j$ in~\eqref{e.b_j=a(xi,pi)=}. Then, for every $\chi\in X(\tau)$ and $\xi\in\Xi_\tau(\chi)$, we have $b_j = \frac{1}{2}(|P_j|-2\chi_j)$ and thus
\begin{align*}
    \sum_{j\in\mathfrak{J}_1}b_jp_j \stackrel{\eqref{e.J_1=,J_0=}}{=} \sum_{j\in \mathfrak{J}_1}b_j = \frac{1}{2}\sum_{j\in\mathfrak{J}_1}|P_j|- \sum_{j\in\mathfrak{J}_1}\chi_j.
\end{align*}
Using this, \eqref{e.Xi_tau(chi)=...}, and~\eqref{e.b_j=a(xi,pi)=}, we have that, for every $\tau\in 4\Z$,
\begin{align*}
    \sum_{\xi\in\Xi_\tau}a(\xi,\pa)&= \sum_{\chi\in X(\tau)}\sum_{\xi\in\Xi_\tau(\chi)}(-1)^{\frac{1}{2}\sum_{j\in\mathfrak{J}_1}|P_j|}(-1)^{\sum_{j\in\mathfrak{J}_1}\chi_j}\nonumber
    \\
    &= (-1)^{\frac{1}{2}\sum_{j\in\mathfrak{J}_1}|P_j|} \sum_{\chi\in X(\tau)} (-1)^{\sum_{j\in\mathfrak{J}_1}\chi_j}\prod_{j=1}^J \genfrac(){0pt}{}{|P_j|}{\chi_j},\end{align*}
where in the last equality we used $|\Xi_\tau(\chi)|=\prod_{j=1}^J \binom{|P_j|}{\chi_j}$. In view of this, to show~\eqref{e.sumsuma=0}, it suffices to show
\begin{align}\label{e.sumsum(-1)...=0}
     \sum_{\tau\in 4\Z\cap[-n,n]}\ \sum_{\chi\in X(\tau)} (-1)^{\sum_{j\in\mathfrak{J}_1}\chi_j}\prod_{j=1}^J \binom{|P_j|}{\chi_j} =0.
\end{align}

To show~\eqref{e.sumsum(-1)...=0}, we use some simple binomial formulas. Multiplying the following terms
\begin{align*}
\begin{cases}
    2^{|P_j|}= (1+1)^{|P_j|} = \sum_{\chi_j=0}^{|P_j|} \binom{|P_j|}{\chi_j},\qquad &\forall j \in \mathfrak{J}_0,
    \\
    0= (1-1)^{|P_j|} = \sum_{\chi_j=0}^{|P_j|} \genfrac(){0pt}{}{|P_j|}{\chi_j}(-1)^{\chi_j},\qquad  &\forall j \in \mathfrak{J}_1,
\end{cases}
\end{align*}
for all $j\in\llbracket 1, J\rrbracket$,
we get
\begin{align}\label{e.0=..J_1}
    0 = \sum_{\chi:\:\chi_j\in\llbracket 0,|P_j|\rrbracket,\,\forall j}(-1)^{\sum_{j\in\mathfrak{J}_1}\chi_j}\prod_{j=1}^J \binom{|P_j|}{\chi_j} \stackrel{\eqref{e.chi=bigcupX(tau)}}{=} \sum_{\tau\in 2\Z\cap[-n,n]}\ \sum_{\chi\in X(\tau)}\cdots
\end{align}
where we omitted the same summand in $\cdots$. 
On the other hand, multiplying
\begin{align*}
\begin{cases}
    0= (1-1)^{|P_j|} = \sum_{\chi_j=0}^{|P_j|} \binom{|P_j|}{\chi_j}(-1)^{\chi_j},\qquad &\forall j \in \mathfrak{J}_0,
    \\
    2^{|P_j|}= (1+1)^{|P_j|} = \sum_{\chi_j=0}^{|P_j|} \binom{|P_j|}{\chi_j},\qquad  &\forall j \in \mathfrak{J}_1,
\end{cases}
\end{align*}
for all $j\in\llbracket 1, J\rrbracket$, we obtain
\begin{align}\label{e.0=..J_0}
    0 = \sum_{\chi:\:\chi_j\in\llbracket 0,|P_j|\rrbracket,\,\forall j}(-1)^{\sum_{j\in\mathfrak{J}_0}\chi_j}\prod_{j=1}^J \binom{|P_j|}{\chi_j} \stackrel{\eqref{e.chi=bigcupX(tau)}}{=} \sum_{\tau\in 2\Z\cap[-n,n]}\ \sum_{\chi\in X(\tau)}\cdots.
\end{align}
We emphasize that it is crucial to have $\mathfrak{J}_0\neq \emptyset $ and $\mathfrak{J}_1\neq \emptyset$ to get~\eqref{e.0=..J_1} and~\eqref{e.0=..J_0} here.

To proceed, we analyze the parities of $\sum_{j\in\mathfrak{J}_0}\chi_j$ and $\sum_{j\in\mathfrak{J}_1}\chi_j$. The definition of $X(\tau)$ in~\eqref{e.X(tau)=} implies 
\begin{align*}
    \chi \in X(\tau)\qquad \Longrightarrow\qquad \sum_{j\in\mathfrak{J}_0}\chi_j + \sum_{j\in\mathfrak{J}_1}\chi_j = \frac{n-\tau}{2} .
\end{align*}
Hence, when $\frac{n}{2}$ is odd,  if $\tau\in 2 + 4\Z$, we have $(-1)^{\sum_{j\in\mathfrak{J}_0}\chi_j} = (-1)^{\sum_{j\in\mathfrak{J}_1}\chi_j}$ and if $\tau\in 4\Z$, then $(-1)^{\sum_{j\in\mathfrak{J}_0}\chi_j}=- (-1)^{\sum_{j\in\mathfrak{J}_1}\chi_j}$. In this case, subtracting~\eqref{e.0=..J_0} from~\eqref{e.0=..J_1}, we obtain~\eqref{e.sumsum(-1)...=0}. When $\frac{n}{2}$ is even, we have $(-1)^{\sum_{j\in\mathfrak{J}_0}\chi_j} = -(-1)^{\sum_{j\in\mathfrak{J}_1}\chi_j}$ if $\tau\in 2 + 4\Z$ and $(-1)^{\sum_{j\in\mathfrak{J}_0}\chi_j}= (-1)^{\sum_{j\in\mathfrak{J}_1}\chi_j}$ if $\tau\in 4\Z$. In this case, adding~\eqref{e.0=..J_0} to~\eqref{e.0=..J_1}, we obtain~\eqref{e.sumsum(-1)...=0}.
\end{proof}

\subsection{Derivation of (\ref{e.GFF_compute}): the GFF computation}

\begin{lemma}\label{l.GFF_test_function}
For every $n\in 2\N$, $\xi\in\Xi_0$, distinct points $x_1,\dots,x_n\in\R^2$, and $0<\eps\leq \min_{i\neq j}\{|x_i-x_j|\}/2$, let $F_{x,\xi,\eps}$ be given as in~\eqref{e.F_xi,eps=} and we have
\begin{align*}
    \lim_{\delta\to0}\phi_{\delta\Z^2}\Ll[\mathcal{A}_{F_{x,\xi,\eps}}(\mcl L)\Rr] & =\exp\left( \frac{1}{2\pi^2}\iint F_{x,\xi,\eps}(\rz)F_{x,\xi,\eps}(\rz')\log|\rz-\rz'|\d\rz\d\rz'\right) \\
    &= \eps^\frac{n}{8}e^{nc_0}\prod_{1\leq i<j\leq n}\Ll(|x_i-x_j|^{\frac{\xi_i\xi_j}{4}}e^{2\xi_i\xi_jc_\eps(x_i-x_j)}\Rr)
\end{align*}
where the convergence is uniform over $(x_1,\dots,x_n)\in(\R^2)^n$ satisfying $\max_i|x_i|\leq 1/\dst$ for any $\dst\in(0,1)$; the constant $c_0$ is given as in Proposition~\ref{p}; and we have set
\begin{align*}
    c_\eps(v)  := \frac{1}{8\pi^2}\iint_{B_1(0)\times B_1(0)}\log\frac{|v+\eps(\rz-\rz')|}{|v|}\d\rz\d\rz',\qquad\text{for every $v\in\R^2\setminus\{0\}$ and $\eps>0$.}
\end{align*}
\end{lemma}

It is easy to see that
\begin{align}\label{e.|c_eps(v)|<}
    |c_\eps(v)| \leq \tfrac{\eps}{4(|v|-2\eps)},\qquad\text{whenever $2\eps < |v|$.}
\end{align}
This along with Lemma~\ref{l.GFF_test_function} gives the first line in~\eqref{e.GFF_compute}.

\begin{proof}
The first equality is a direct consequence of Theorem~\ref{thm:M formula}. It remains to compute the ensuing integral.
We write $F_{x,\xi,\eps} =\sum_{i=1}^n f_i$ for $f_i=\frac{\xi_i}{2\eps^2}\one_{B_\eps(x_i)}$ and $\la f,g\ra = \frac{1}{2\pi^2}\iint f(\rz)g(\rz')\log|z-z'|\d \rz\d \rz'$ for functions $f,g$. 
We have $\la F_{x,\xi,\eps},F_{x,\xi,\eps}\ra=\sum_{i,j=1}^n\la f_i,f_j \ra$. 
We can easily compute 
\begin{align*}
    \la f_i,f_i \ra=\frac{1}{8}\log\eps+c_0 \quad \text{and}\quad \la f_i,f_j\ra = \xi_i\xi_j\left(\frac{1}{8}\log|x_i-x_j|+c_\eps(x_i-x_j)\right) \quad \text{when $i\neq j$}.
\end{align*}
Inserting these into $\exp(\la F_{x,\xi,\eps},F_{x,\xi,\eps}\ra)$ gives the desired result. The uniformity in $(x_i)_{i\in\llbracket 1,n\rrbracket}$ is ensured by Lemma~\ref{lem:uniform_gff_convergence}.
\end{proof}

The second case in~\eqref{e.GFF_compute} follows from the lemma below.

\begin{lemma}\label{l.e.GFF_compute}
For every $n\in 2\N$, $\tau\in 4\Z\setminus\{0\}$, $\xi\in\Xi_\tau$, distinct points $x_1,\dots,x_n\in\R^2$, and $0<\eps\leq \min_{i\neq j}\{|x_i-x_j|\}/2$, we have $\lim_{\delta\rightarrow 0} \phi_{\delta\Z^2}\Ll[\mathcal{A}_{F_{x,\xi,\eps}}(\mcl L)\Rr] = 0$, where the convergence is uniform over $(x_1,\dots,x_n)\in(\R^2)^n$ satisfying $\max_i|x_i|\leq 1/\dst$ for any $\dst\in(0,1)$.
\end{lemma}

\begin{proof}
We write $F=F_{x,\xi,\eps}$ for brevity.
Let $R>0$ and consider $\calL_{\pm R}$ consisting of loops contained in $B_{\sqrt{R}}(\pm R,0)$. Set $u_R:=(R,0)$, define $F_R := F(\,\cdot\,-u_R)- F(\,\cdot\,+u_R)$, and we apply Lemma~\ref{l.GFF_test_function} to $F_R$ to see that, for $R$ sufficiently large,
\begin{align}
    \lim_{\delta\to0}\phi_{\delta\Z^2}\Ll[\mathcal{A}_{F_R}(\mcl L)\Rr]  &=\eps^\frac{2n}{8}
    e^{2nc_0}\prod_{1\leq i<j\leq n}|x_i-x_j|^{\frac{\xi_i\xi_j}{2}}e^{4\xi_i\xi_jc_\eps(x_i-x_j)} \notag
    \\
    &\cdot\prod_{i,j=1}^n|x_i-x_j-2u_R|^{-\frac{\xi_i\xi_j}{4}}e^{-2\xi_i\xi_jc_\eps(x_i-x_j-2u_R)} = O\Big(\prod_{i,j=1}^nR^{-\frac{\xi_i\xi_j}{4}}\Big) = O\Big(R^{-\frac{k_+-k_-}{4}}\Big) \label{e.limphi=0}
\end{align}
where $k_\pm = \#\{(i,j):\:\xi_i\xi_j=\pm1\}$. Setting $m_\pm=\#\Ll\{i:\:\xi_i=\pm1\Rr\}$, we have $m_\pm = \frac{n\pm\tau}{2}$. Since $k_+ = m_+^2+m_-^2$ and $k_-=2m_+m_-$, we have $k_+-k_-=(m_+-m_-)^2=\tau^2>0$. Therefore, the above vanishes when $R\to\infty$.

On the other hand, using the mixing property (see~\cite[Proposition~2.9]{duminil2022planar}) and the fact that the probability that a loop intersects both~$B_{\sqrt R}(-R,0)$ and~$B_{\sqrt R}(R,0)$ tends to~$0$ as~$R \to \infty$, we can deduce
\begin{align*}
    \phi_{\delta }[\mathcal A_{F}(\calL )]^2
    =\lim_{R\to \infty}\phi_\delta[\mathcal A_{ F(\cdot-(R,0))}(\calL )\mathcal A_{ F(\cdot+(R,0))}(\calL )]
    =\lim_{R\to \infty}\phi_\delta[\mathcal A_{ F_R}(\calL)].
\end{align*}
Furthermore, the convergences above are uniform in~$\delta$, since both the mixing property and the bound on the existence of large loops are uniform in~$\delta$. This allows us to interchange the limits and get
\begin{align*}
    \lim_{\delta\to0}\phi_{\delta }[\mathcal A_{ F}(\calL )]^2
    =\lim_{\delta\to0}\lim_{R\to \infty} \phi_\delta[\mathcal A_{ F_R}(\calL)] 
    =\lim_{R\to \infty}\lim_{\delta\to0} \phi_\delta[\mathcal A_{ F_R}(\calL)] 
    \stackrel{\eqref{e.limphi=0}}{=} 0
\end{align*}
as desired.
The uniformity in $(x_i)_{i\in\llbracket 1,n\rrbracket}$ is ensured by Lemma~\ref{lem:uniform_gff_convergence}.
\end{proof}

\subsection{Derivation of (\ref{e.f_n,d_ratio}): comparison of normalizations}
It is ensured by the lemma below.

\begin{lemma}\label{l.f_n,d}
For $n\in2\N$ and $\dst\in(0,1)$, let $f_{n,\dst}$ be given as in~\eqref{e.cvg_f_n,d}. Then, $\lim_{\delta\to0} \frac{f_{n,\dst}(\delta)}{f_{2,\frac{1}{2}}(\delta)}=1$.
\end{lemma}

\begin{proof}
Fix $n\in2\N$ and choose $d_n>0$ such that $X(n,d_n)$ is nonempty (see~\eqref{e.X(n,d)=}). Set
$f_n(\delta):=f_{n,d_n}(\delta)$. If $X(n,d)$ is nonempty, then choosing
$x\in X(n,d)\cap X(n,d_n)$ and applying~\eqref{e.cvg_f_n,d} with both
$f_{n,d}$ and $f_{n,d_n}$ gives
\begin{align}\label{e.limf/f=1}
    \lim_{\delta\to0}\frac{f_{n,d}(\delta)}{f_n(\delta)}=1.
\end{align}
It remains to show
\begin{align}\label{e.f_n_ratio}
    \lim_{\delta\to0} \frac{f_{n}(\delta)}{f_{2}(\delta)}=1 \qquad\text{for every $n\in 2\N$}.
\end{align}

Let $m,n\in 2\N$. Fix any distinct points $x=(x_i)_{i\in\llbracket1,m\rrbracket}$ and $y=(y_i)_{i\in\llbracket1,n\rrbracket}$. For $R>0$, we set $y_R= (y_i+(R,0))_{i\in\llbracket1,n\rrbracket}$. Slightly abusing the notation, we also denote by $x$ and $y_R$ the set of these points, respectively. For a finite subset $Z$ of $\R^2$, we introduce the notation $\S_{(Z)_\delta} = \prod_{z\in Z} \S_{z_\delta}$. Hence, we write $\S_{(x)_\delta} = \S_{(x_1)_\delta}\cdots \S_{(x_m)_\delta}$ and similarly for $\S_{(y_R)_\delta}$ and $\S_{(x\cup y_R)_\delta}$. We want to show
\begin{align}
    \lim_{R\to\infty}\lim_{\delta\to0}\frac{f_{m+n}^{-m-n}(\delta)\E_\delta\Ll[ \S_{(x\cup y_R)_\delta} \Rr]}{f_{m}^{-m}(\delta)f_{n}^{-n}(\delta)\E_\delta\Ll[ \S_{(x)_\delta} \Rr]\E_\delta\Ll[ \S_{(y_R)_\delta} \Rr]} = 1, \label{e.lim_Rlim_deltaf<S>/ff<S><S>=1}
    \\
    \lim_{R\to\infty}\limsup_{\delta\to0}\Ll|\frac{\E_\delta\Ll[ \S_{(x\cup y_R)_\delta} \Rr]}{\E_\delta\Ll[ \S_{(x)_\delta} \Rr]\E_\delta\Ll[ \S_{(y_R)_\delta} \Rr]} - 1\Rr| = 0. \label{e.lim_Rlim_delta<S>/<S><S>=1}
\end{align}
Let us postpone their verifications and use them to deduce the result. Combining the above two displays, we can get
\begin{align*}
    \lim_{\delta\to0}\frac{f_{m+n}^{-m-n}(\delta)}{f_{m}^{-m}(\delta)f_{n}^{-n}(\delta)} = 1,\qquad\forall m,n\in2\N.
\end{align*}
Then, \eqref{e.f_n_ratio} follows by a simple induction argument. This completes the proof, modulo the verification of~\eqref{e.lim_Rlim_deltaf<S>/ff<S><S>=1} and~\eqref{e.lim_Rlim_delta<S>/<S><S>=1}, which we carry out below.
\end{proof}

\begin{proof}[Proof of \eqref{e.lim_Rlim_deltaf<S>/ff<S><S>=1}]
For each fixed $R$, we first apply~\eqref{e.cvg_f_n,d} with some $d_R>0$
such that $x\cup y_R\in X(m+n,d_R)$, and then replace $f_{m+n,d_R}$ by
$f_{m+n}$ using~\eqref{e.limf/f=1}.
We write $\xi^x\in\{-1,1\}^m$,  $\xi^y\in\{-1,1\}^n$, and $\xi=(\xi^x,\xi^y)\in\{-1,1\}^{m+n}$. 
The right-hand side in~\eqref{e.cvg_f_n,d} to $\lim_{\delta\rightarrow 0}f_{m+n}^{-m-n}(\delta)\E_\delta\Ll[ \S_{(x\cup y_R)_\delta}\Rr]$ is the sum over $\xi=(\xi^x,\xi^y)\in\{-1,1\}^{m+n}$ with $\sum_i\xi_i=0$ of
\begin{equation}\label{e.contri_uneq}
    2^{-m-n}e^{c_0(m+n)}\prod_{1\leq i<j\leq m+n}|z_i-z_j|^{\frac{\xi_i\xi_j}{4}}
\end{equation}
where we enumerate $x\cup y_R$ as $z_1,\dots,z_m,z_{m+1},\dots,z_{m+n}$ with $z_i=x_i$ for $i \in\llbracket 1,m\rrbracket$ and $z_{m+j} = y_j + (R,0)$ for $j \in\llbracket 1,n\rrbracket$.
The product in \eqref{e.contri_uneq} can be further decomposed as
\begin{equation*}\prod_{1\leq i<j\leq m+n}|z_i-z_j|^{\frac{\xi_i\xi_j}{4}}=\prod_{\substack{1\leq i\leq m\\ 1\leq j\leq n}}|x_i-y_j-(R,0)|^{\frac{\xi^x_i\xi^y_j}{4}}\bigg(\prod_{1\leq i<j\leq m}|x_i-x_j|^{\frac{\xi^x_i\xi^x_j}{4}}\prod_{1\leq i<j\leq n}|y_i-y_j|^{\frac{\xi^y_i\xi^y_j}{4}}\bigg).
\end{equation*}
It is easy to see that the products inside the parentheses are independent of $R$. Only the product outside the parentheses depends on $R$, and we now analyze this term.

Define $n_{xy}:=\sum_{i=1}^m\xi^x_i$ and we have $-n_{xy}=\sum_{i=1}^n\xi^y_i$ since the sum of entries in $\xi$ is zero.
Due to $\frac{|x_i-y_j-(R,0)|}{R}=1+o(1)$, we get
\begin{equation*}
    \prod_{\substack{1\leq i\leq m\\ 1\leq j\leq n}}|x_i-y_j-(R,0)|^{\frac{\xi^x_i\xi^y_j}{4}} = R^{-\frac{n_{xy}^2}{4}}(1+o(1)).
\end{equation*}
When $n_{xy}\neq 0$, we can see that the term in~\eqref{e.contri_uneq} vanishes as $R\to\infty$.
When $n_{xy}=0$, we have $\sum_{i=1}^m \xi_i^x=0$ and $\sum_{i=1}^n\xi^y_i=0$, and the term in~\eqref{e.contri_uneq} can be written as
\begin{equation*}
    (1+o_R(1))\bigg(2^{-m}e^{c_0m}\prod_{1\leq i<j\leq m}|x_i-x_j|^{\frac{\xi^x_i\xi^x_j}{4}}\bigg)\bigg(2^{-n}e^{c_0n}\prod_{1\leq i<j\leq n}|y_i-y_j|^{\frac{\xi^y_i\xi^y_j}{4}}\bigg),
\end{equation*}
which is exactly the product of the contributions as in~\eqref{e.cvg_f_n,d} of $\xi^x$ and $\xi^y$ to $f_{m}^{-m}(\delta)\E_\delta\Ll[ \S_{(x)_\delta}\Rr]$ and $f_{n}^{-n}(\delta)\E_\delta\Ll[ \S_{(y)_\delta}\Rr]$, respectively.
Therefore, we can deduce
\begin{align*}
\lim_{R\rightarrow \infty}\lim_{\delta\rightarrow 0}f_{m+n}^{-m-n}(\delta)\E_\delta\Ll[ \S_{(x\cup y_R)_\delta}\Rr]
&= \lim_{\delta\rightarrow 0}f_{m}^{-m}(\delta)\E_\delta\Ll[ \S_{(x)_\delta}\Rr] f_{n}^{-n}(\delta)\E_\delta\Ll[ \S_{(y)_\delta}\Rr]
\\
&= \lim_{R\rightarrow \infty}\lim_{\delta\rightarrow 0}f_{m}^{-m}(\delta)\E_\delta\Ll[ \S_{(x)_\delta}\Rr] f_{n}^{-n}(\delta)\E_\delta\Ll[ \S_{(y_R)_\delta}\Rr],
\end{align*}
where the last equality follows from the fact that $y_R$ is a translation of $y$.
Since the above limits are nonzero, we can conclude~\eqref{e.lim_Rlim_deltaf<S>/ff<S><S>=1}.
\end{proof}

\begin{proof}[Proof of~\eqref{e.lim_Rlim_delta<S>/<S><S>=1}]    
Recall from~\eqref{e.<S..S>=sum_P} the expansion of the spin correlation into a summation of connection probabilities. For sets $x$, $y_R$, and $x\cup y_R$, we denote by $\mathfrak{P}_x$, $\mathfrak{P}_{y_R}$, and $\mathfrak{P}_{x\cup y_R}$, respectively, the associated collection $\mathfrak{P}_{+,1}$ of even partitions appearing in~\eqref{e.<S..S>=sum_P}. Each partition $\pa$ in these collections gives rise to the event $E_0(\pa)$ (defined below~\eqref{e.frakP_+=}) that points are connected in open clusters according to $\pa$. Hence, in the current setting, using the shorthand $\phi=\phi_{\delta\Z^2}$, we obtain from~\eqref{e.<S..S>=sum_P} that
\begin{align*}
    \E_\delta\Ll[ \S_{(x)_\delta}\Rr] =\sum_{\pa_1 \in \mathfrak{P}_x}\phi[E_0(\pa_1)],\qquad \E_\delta\Ll[ \S_{(y_R)_\delta}\Rr] =\sum_{\pa_2 \in \mathfrak{P}_{y_R}}\phi[E_0(\pa_2)],
    \\
    \text{and}\qquad \E_\delta\Ll[ \S_{(x\cup y_R)_\delta}\Rr] =\sum_{\pa \in \mathfrak{P}_{x\cup y_R}}\phi[E_0(\pa)].
\end{align*}
Each $(\pa_1,\pa_2)\in\mathfrak{P}_x\times\mathfrak{P}_{y_R}$ uniquely determines a partition in $\mathfrak{P}_{x\cup y_R}$. We write $E_0(\pa_1,\pa_2)$ for the corresponding event. Thus, we have
\begin{align}
    \E_\delta\Ll[ \S_{(x\cup y_R)_\delta}\Rr] - \E_\delta\Ll[ \S_{(x)_\delta}\Rr] \E_\delta\Ll[ \S_{(y_R)_\delta}\Rr] 
    =\sum_{(\pa_1,\pa_2)\in\mathfrak{P}_x\times\mathfrak{P}_{y_R}}\Ll(\phi\Ll[E_0(\pa_1,\pa_2)\Rr] -\phi\Ll[E_0(\pa_1)\Rr]\phi\Ll[ E_0(\pa_2)\Rr] \Rr) \notag
    \\
    + \sum_{\pa \in \mathfrak{P}_{x\cup y_R}\setminus\mathfrak{P}_x\times\mathfrak{P}_{y_R}}\phi\Ll[E_0(\pa)\Rr]. \label{e.<S>-<S><S>=}
\end{align}
In the following, we estimate the sums on the right.
Set $d= \frac{1}{8} \min\{ \min_{i\neq i'}|x_i-x_{i'}|,\ \min_{j\neq j'}|y_j-y_{j'}|\}$ and set $D=8\max_{i,j}\{|x_i|,|y_j|\}$.
We use $\lesssim$ to omit a constant independent of $\delta, R$.

Let $\pa \in \mathfrak{P}_{x\cup y_R}\setminus\mathfrak{P}_x\times\mathfrak{P}_{y_R}$. 
On $E_0(\pa)$, there is a one-arm event from $0$ to $d$ (in Euclidean distance) centered at every point in $x\cup y_R$, and there is a one-arm crossing $[-R/2,R/2]^2\setminus [-D,D]^2$ for sufficiently large $R$ (since there has to be a primal connection between some $x_i$ and $(y_R)_j$ due to $\pa \notin\mathfrak{P}_x\times\mathfrak{P}_{y_R}$). 
In addition, these events are well-separated. Therefore, we can use the mixing property (see~\cite[Proposition~2.9]{duminil2022planar}) to get
\begin{align*}
    \phi[E_0(\pa)] \lesssim \pi_1(d/\delta)^{m+n}\pi_1(D/\delta, R/2\delta) \lesssim \pi_1(d/\delta)^{m+n} R^{-c_1}, \qquad\forall \pa \in \mathfrak{P}_{x\cup y_R}\setminus\mathfrak{P}_x\times\mathfrak{P}_{y_R}
\end{align*}
for some absolute constant $c_1>0$.
Similarly, since there are one-arm events from $0$ to $d$ at each point, by the standard RSW-type argument, we have
\begin{align}\label{e.phi[E_0(P_1)]>}
    \pi_1(d/\delta)^m\lesssim\phi[E_0(\pa_1)] \quad\text{and}\quad \pi_1(d/\delta)^n\lesssim\phi[E_0(\pa_2)],\qquad\forall (\pa_1,\pa_2) \in\mathfrak{P}_x\times\mathfrak{P}_{y_R}.
\end{align}
Therefore, we obtain
\begin{align}\label{e.2nd_sum_in<S>-<S><S>}
     \sum_{\pa \in \mathfrak{P}_{x\cup y_R}\setminus\mathfrak{P}_x\times\mathfrak{P}_{y_R}}\phi\Ll[E_0(\pa)\Rr] \lesssim \pi_1(d/\delta)^{m+n} R^{-c_1} \lesssim R^{-c_1} \E_\delta\Ll[ \S_{(x)_\delta}\Rr] \E_\delta\Ll[ \S_{(y_R)_\delta}\Rr].
\end{align}

Next, to estimate the first sum on the right of~\eqref{e.<S>-<S><S>=}, let $(\pa_1,\pa_2)\in\mathfrak{P}_x\times\mathfrak{P}_{y_R}$. Set $Q_1=[-\sqrt{R},\sqrt{R}]^2$ and $Q_2=(R,0)+[-\sqrt{R},\sqrt{R}]^2$. Let $G_1$ (resp.\ $G_2$) be the event that every open cluster containing one of the points $(x_i)_\delta$ (resp.\ $(y_j+(R,0))_\delta$) is contained in $Q_1$ (resp.\ $Q_2$). Then, $E_0(\pa_i)\cap G_i$ is determined by the configuration in $Q_i$, and
\begin{align*}
    E_0(\pa_1,\pa_2)\cap G_1\cap G_2
    =
    E_0(\pa_1)\cap E_0(\pa_2)\cap G_1\cap G_2.
\end{align*}
Thus, we can bound each summand in the first sum in~\eqref{e.<S>-<S><S>=} as
\begin{align}
    &\Ll|\phi\Ll[E_0(\pa_1,\pa_2)\Rr] -\phi\Ll[E_0(\pa_1)\Rr]\phi\Ll[ E_0(\pa_2)\Rr]\Rr| \notag
    \\
    &\leq \Ll|\phi\Ll[E_0(\pa_1)\cap E_0(\pa_2)\cap G_1\cap G_2\Rr] -\phi\Ll[E_0(\pa_1)\cap G_1\Rr]\phi\Ll[ E_0(\pa_2)\cap G_2\Rr]\Rr| \label{e.|phi-phiphi|_line_1}
    \\
    &+ \phi\Ll[E_0(\pa_1,\pa_2)\cap G_1^\comple\Rr] + \phi\Ll[E_0(\pa_1,\pa_2)\cap G_2^\comple\Rr] \label{e.|phi-phiphi|_line_2}
    \\
    &+ \phi\Ll[E_0(\pa_1)\cap G_1^\comple\Rr]\phi\Ll[ E_0(\pa_2)\Rr] + \phi\Ll[E_0(\pa_1)\Rr]\phi\Ll[ E_0(\pa_2)\cap G_2^\comple\Rr]. \label{e.|phi-phiphi|_line_3}
\end{align}
By the mixing property (see~\cite[Proposition~2.9]{duminil2022planar}), there is an absolute constant $c_2>0$ such that
\begin{align*}
    \text{the term in~\eqref{e.|phi-phiphi|_line_1}}
    \leq
    R^{-c_2}\phi[E_0(\pa_1)\cap G_1]\phi[E_0(\pa_2)\cap G_2]
    \leq
    R^{-c_2}\phi[E_0(\pa_1)]\phi[E_0(\pa_2)].
\end{align*}
On $E_0(\pa_1,\pa_2)\cap G_1^\comple$ (resp.\ $E_0(\pa_1,\pa_2)\cap G_2^\comple$), apart from one-arm events centered at each point giving rise to $\pi_1(d/\delta)$, there is also a one-arm crossing $[-\sqrt{R},\sqrt{R}]^2\setminus[-D,D]^2$ (resp.\ the set shifted by $(R,0)$). Hence, the mixing property gives
\begin{align*}
    \text{each term in~\eqref{e.|phi-phiphi|_line_2}}
    \lesssim
    \pi_1(d/\delta)^{m+n}\pi_1\Ll(D/\delta,\sqrt{R}/\delta\Rr)
    \stackrel{\eqref{e.phi[E_0(P_1)]>}}{\lesssim}
    R^{-c_3}\phi[E_0(\pa_1)]\phi[E_0(\pa_2)]
\end{align*}
for some constant $c_3>0$. The two terms in~\eqref{e.|phi-phiphi|_line_3} are treated similarly.
Combining these, we obtain that
\begin{align*}
    \Ll|\text{the first sum in~\eqref{e.<S>-<S><S>=}}\Rr|
    \lesssim R^{-c_4}\sum_{(\pa_1,\pa_2)\in\mathfrak{P}_x\times\mathfrak{P}_{y_R}}\phi[E_0(\pa_1)]\phi[E_0(\pa_2)] = R^{-c_4}\E_\delta\Ll[ \S_{(x)_\delta}\Rr] \E_\delta\Ll[ \S_{(y_R)_\delta}\Rr] 
\end{align*}
for some constant $c_4>0$. Inserting this and~\eqref{e.2nd_sum_in<S>-<S><S>} into~\eqref{e.<S>-<S><S>=} gives
\begin{align*}
    \Ll|\E_\delta\Ll[ \S_{(x\cup y_R)_\delta}\Rr] - \E_\delta\Ll[ \S_{(x)_\delta}\Rr] \E_\delta\Ll[ \S_{(y_R)_\delta}\Rr]\Rr|\lesssim R^{-c_5}\E_\delta\Ll[ \S_{(x)_\delta}\Rr] \E_\delta\Ll[ \S_{(y_R)_\delta}\Rr] 
\end{align*}
for some constant $c_5>0$. Taking limits, we can deduce~\eqref{e.lim_Rlim_delta<S>/<S><S>=1}. 
\end{proof}

\section{Proof of Theorem~\ref{t.Phi}}

To identify the limit of the magnetization field \(\Phi_\delta\) defined in~\eqref{e.Phi_delta=}, we compute the limits of its moments and show that these moments uniquely determine the limit. We begin with some estimates. Recall $\varrho(\delta)$ from~\eqref{e.varrho=}.

\subsection{Second moments of the magnetization}

We introduce the mean magnetization
\begin{align}\label{e.m_delta=}
    m_\delta := \int \Phi_\delta \one_{[0,1]^2} \stackrel{\eqref{e.Phi_delta=}}{=} \delta^2\varrho(\delta)^{-\frac{1}{2}}\sum_{x\in\delta\Z^2\cap[0,1]^2}\S_x.
\end{align}
Since the correlation of $\S$ among an odd number of points is zero and $\S_x^2 =1$, we have
\begin{align}\label{e.<m^n>=0_odd}
    \E_\delta\Ll[ m^n_\delta\Rr] =0,\qquad\forall n\notin 2\N.
\end{align}

\begin{lemma}\label{l.<m^2><}
We have $\limsup_{\delta\to0} \E_\delta\Ll[ m^2_\delta\Rr]<\infty$.
\end{lemma}

\begin{proof}
Due to~\eqref{e.<S..S>=sum_P}, $\E_\delta\Ll[ \S_x\S_y\Rr]$ corresponds to the probability under $\phi_{\delta\Z^2}$ that there is a primal connection between $x$ and $y$. Hence, by a standard RSW-type argument, we have
\begin{align}\label{e.<SS>=pi_1^2}
    \E_\delta\Ll[ \S_x\S_y\Rr]\asymp \pi_1\Ll(\delta^{-1}|x-y|_\infty\Rr)^2
\end{align}
where $|z|_\infty= \max \{|z_1|,\,|z_2|\}$. In particular, we have
\begin{align}\label{e.f=pi_1}
    \varrho(\delta)^\frac{1}{2} \stackrel{\eqref{e.varrho=}}{\asymp} \E_\delta\Ll[  \S_{(0,0)_\delta} \S_{(1,0)_\delta}\Rr]^\frac{1}{2} \stackrel{\eqref{e.<SS>=pi_1^2}}{\asymp} \pi_1(\delta^{-1}).
\end{align}
By the quasi-multiplicativity of $\pi_1$, there are constants $C>0$ and $c\in(0,1)$ such that
\begin{align}\label{e.pi_1/pi_2<}
    \pi_1(2^i)/\pi_1(\delta^{-1}) \leq C\Ll(2^i\delta\Rr)^{-c}\qquad\text{for $2^i\leq \delta^{-1}$.}
\end{align}
We can expand
\begin{align*}
    \delta^{-4} \varrho(\delta)\E_\delta\Ll[ m_\delta^2\Rr] \stackrel{\eqref{e.m_delta=}}{=} \sum_{x,y\in\delta\Z^2\cap[0,1]^2}\E_\delta\Ll[ \S_x\S_y\Rr] \stackrel{\eqref{e.<SS>=pi_1^2}}{\asymp}  \sum_{x,y\in\delta\Z^2\cap[0,1]^2}\pi_1\Ll(\delta^{-1}|x-y|_\infty\Rr)^2
    \\
    \lesssim \sum_{y\in\delta\Z^2\cap[0,1]^2}\sum_{i=0}^{\log_2(1/\delta)}\sum_{\delta 2^i \leq \|x\|_\infty \leq \delta 2^{i+1}} \pi_1(2^i)^2  \lesssim \delta^{-2}\sum_{i=0}^{\log_2(1/\delta)}2^{2i} \pi_1(2^i)^2
    \\
    \stackrel{\eqref{e.pi_1/pi_2<}}{\lesssim} \delta^{-2-2c}\pi_1(\delta^{-1})^2\sum_{i=1}^{\log_2(1/\delta)} 2^{(2-2c)i} \lesssim \delta^{-4} \pi_1(\delta^{-1})^2 \stackrel{\eqref{e.f=pi_1}}{\asymp} \delta^{-4}\varrho(\delta)
\end{align*}
which gives the desired result.
\end{proof}

We record a useful bound. For any bounded continuous $h$ on $[0,1]^2$, we have
\begin{align}
     \Ll|\E_\delta\Ll[ \Ll(\int \Phi_\delta h \Rr)^n\Rr]\Rr| &\stackrel{\eqref{e.Phi_delta=}}{=} \delta^{2n} \varrho(\delta)^{-\frac{n}{2}} \E_\delta\Ll[ \sum_{x_1,\dots,x_n\in\delta\Z^2\cap[0,1]^2} h(x_1)\cdots h(x_n)\S_{x_1}\cdots \S_{x_n} \Rr] \notag
     \\
     &\leq  \delta^{2n} \varrho(\delta)^{-\frac{n}{2}}\|h\|_\infty^n\sum_{x_1,\dots,x_n\in\delta\Z^2\cap[0,1]^2}\E_\delta\Ll[\S_{x_1}\cdots \S_{x_n}\Rr] \notag
    \\
    &\stackrel{\eqref{e.m_delta=}}{=}  \|h\|_\infty^n \E_\delta\Ll[ m_\delta^n\Rr]. \label{e.intPhih<hm}
\end{align}

\subsection{Sobolev spaces and tightness}\label{s.sobolev}

Recall \(\N=\{1,2,\ldots\}\). On \([0,1]^2\), we use the Dirichlet sine basis
\(e_{jk}(z)=2\sin(j\pi z_1)\sin(k\pi z_2)\), \(j,k\in\N\). This is an
orthonormal basis of \(L^2([0,1]^2)\). For \(h\in L^2([0,1]^2)\), write
\(\hat h_{jk}=\int_{[0,1]^2}h e_{jk}\). For \(s\ge0\), let
\(\mcl H^s_0([0,1]^2)\) and \(\mcl H^{-s}([0,1]^2)\) be completions of
\(L^2([0,1]^2)\) under the norms
\[
    \|h\|^2_{\mcl H^s_0([0,1]^2)}
    =
    \sum_{j,k\in\N} \Ll(\hat h_{jk}\Rr)^2\Ll(j^2+k^2\Rr)^s\qquad\text{and}\qquad
    \|h\|^2_{\mcl H^{-s}([0,1]^2)}
    =
    \sum_{j,k\in\N} \Ll(\hat h_{jk}\Rr)^2\Ll(j^2+k^2\Rr)^{-s},
\]
respectively.
With the pairing
\(\langle f,h\rangle=\sum_{j,k\in\N}\hat f_{jk}\hat h_{jk}\),
\(\mcl H^{-s}([0,1]^2)\) is naturally identified with the dual of
\(\mcl H^s_0([0,1]^2)\).

For \(L>0\), the spaces \(\mcl H^s_0([-L,L]^2)\) and
\(\mcl H^{-s}([-L,L]^2)\) are defined by translating and dilating the
above construction.

On \(\R^2\), we use the corresponding local spaces. Namely,
\(\mcl H^s_0(\R^2)\) denotes the space of compactly supported functions
\(h\) such that, for some \(L\in\N\) with
\(\operatorname{supp} h\subset (-L,L)^2\), the restriction
\(h|_{[-L,L]^2}\) belongs to \(\mcl H^s_0([-L,L]^2)\). We use the natural
compact-support topology: \(h_n\to h\) if all supports are eventually
contained in a common box \((-L,L)^2\) and
\(h_n\to h\) in \(\mcl H^s_0([-L,L]^2)\).

The dual of \(\mcl H^s_0(\R^2)\) is the local negative Sobolev space consisting of every distribution $u$ on $\R^2$ satisfying $u|_{[-L,L]^2}\in \mcl H^{-s}([-L,L]^2)$ for every $L\in\N$,
equipped with the topology generated by the semi-norms
\( \|u|_{[-L,L]^2}\|_{\mcl H^{-s}([-L,L]^2)}\) for \(L\in\N\). 

\begin{lemma}\label{l.tightness}
For every \(L>0\) and \(s>1\), we have $\limsup_{\delta\to0}\E_\delta\Ll[ \|\Phi_\delta\|^2_{\mcl H^{-s}([-L,L]^2)} \Rr] <\infty$.
Consequently, \((\Phi_\delta)_{\delta>0}\) is tight in
\(\mcl H^{-s}(\R^2)\) for every \(s>1\).
\end{lemma}

\begin{proof}
It suffices to prove the estimate on \([0,1]^2\). The proof on
\([-L,L]^2\) is the same after translation and dilation. By the definition
of the \(\mcl H^{-s}([0,1]^2)\)-norm and Tonelli's theorem,
\begin{align*}
    \E_\delta\Ll[ \|\Phi_\delta\|^2_{\mcl H^{-s}([0,1]^2)} \Rr]
    =
    \sum_{j,k\in\N}
    \E_\delta\Ll[
    \Ll( \int_{[0,1]^2} \Phi_\delta e_{jk} \Rr)^2
    \Rr]
    \Ll(j^2+k^2\Rr)^{-s}                                      \stackrel{\eqref{e.intPhih<hm}}{\le}
    4\E_\delta\Ll[ m_\delta^2\Rr]
    \sum_{j,k\in\N}\Ll(j^2+k^2\Rr)^{-s}.
\end{align*}
The last sum is finite for \(s>1\), and
Lemma~\ref{l.<m^2><} gives a uniform bound on \(\E_\delta\Ll[ m_\delta^2\Rr]\).
This proves the finite-volume estimate.

For tightness, fix \(s>1\) and choose \(s'\in(1,s)\). The estimate above
gives a uniform bound in \(\mcl H^{-s'}([-L,L]^2)\) for every \(L\). Since
the embedding $\mcl H^{-s'}([-L,L]^2)\hookrightarrow\mcl H^{-s}([-L,L]^2)$ is compact, the restrictions \((\Phi_\delta|_{[-L,L]^2})_{\delta>0}\) are
tight in \(\mcl H^{-s}([-L,L]^2)\) for every fixed \(L\). Since the
topology of \(\mcl H^{-s}(\R^2)\) is generated by these finite-volume
semi-norms, tightness of all restrictions implies tightness in
\(\mcl H^{-s}(\R^2)\).
\end{proof}

\subsection{Identifying the limit via moments}

We need the following two ingredients.
Recall $M_n(h)$ from~\eqref{e.M_n(h)=}.

\begin{lemma}\label{l.mom_Phi_delta}
For every $s>1$ and every $h\in \mcl H_0^s(\R^2)$, we have $\lim_{\delta\to0} \E_\delta\Ll[\Ll(\int \Phi_\delta h\Rr)^n\Rr] = M_n(h)$ for every $n\in\N$.
\end{lemma}

\begin{proof}
Without loss of generality, we can assume that $h$ is supported on $[0,1]^2$. The case $n\notin2\N$ follows from~\eqref{e.<m^n>=0_odd} and~\eqref{e.intPhih<hm}. Now, let $n\in 2\N$, fix any $\eps>0$, and expand
\begin{align}
    \E_\delta\Ll[\Ll(\int \Phi_\delta h\Rr)^n\Rr] &\stackrel{\eqref{e.Phi_delta=}}{=}\sum_{x_1,\dots, x_n \in \delta\Z^2\cap [0,1]^2}(\delta^2\varrho(\delta)^{-\frac{1}{2}})^n h(x_1)\cdots h(x_n)\E_\delta\Ll[ \S_{x_1}\cdots\S_{x_n}\Rr] \notag
    \\
    &=\sum_{\substack{x_1,\dots, x_n \in \delta\Z^2\cap [0,1]^2 \\ \inf_{i\neq j}|x_i-x_j|\geq\eps}}\delta^{2n} h(x_1)\cdots h(x_n)\varrho(\delta)^{-\frac{n}{2}}\E_\delta\Ll[ \S_{x_1}\cdots\S_{x_n}\Rr] \label{e.<Phih>^n=_I}
    \\
    &+ \sum_{\substack{x_1,\dots, x_n \in \delta\Z^2\cap [0,1]^2 \\ \inf_{i\neq j}|x_i-x_j|<\eps}} (\delta^2\varrho(\delta)^{-\frac{1}{2}})^n h(x_1)\cdots h(x_n)\E_\delta\Ll[ \S_{x_1}\cdots\S_{x_n}\Rr] .\label{e.<Phih>^n=_II}
\end{align}
Theorem~\ref{t} implies that
\begin{align*}
    \lim_{\delta\to0} \varrho(\delta)^{-\frac{n}{2}}\E_\delta\Ll[ \S_{x_1}\cdots\S_{x_n}\Rr]  = \la x_1,\dots, x_n\ra 
\end{align*}
uniformly in $x_1,\dots,x_n\in\delta\Z^2\cap[0,1]^2$ with $\inf_{i\neq j}|x_i-x_j|\geq \eps$.
Since $h \in \mcl H^s_0(\R^2)$ and $s>1$, $h$ is uniformly continuous. Using these, we can see that, for each $\eps>0$, the sum in~\eqref{e.<Phih>^n=_I} converges as $\delta\to0$ to
\begin{align*}
    \int_{\substack{x_1,\dots, x_n \in  [0,1]^2 \\ \inf_{i\neq j}|x_i-x_j|\geq\eps}} h(x_1)\cdots h(x_n)\la x_1,\dots,x_n\ra \d x_1\cdots \d x_n.
\end{align*}
By adapting the arguments in~\cite[Section~3.2]{camia2015planar} together with the modification in~\cite[Section~4.2]{camia2015planar}\footnote{The proof in \cite[Section~3.2]{camia2015planar} is based on the random cluster representation and induction. One key ingredient for adaptation is that $\varrho(\delta)^\frac{1}{2}$ is comparable to the one-arm probability as in~\eqref{e.f=pi_1}. Our setting is closer to the description in~\cite[Section~4.2]{camia2015planar}.}, we can show that there are constants $C,c>0$ such that, uniformly in $0<\delta<\eps$, 
\begin{align*}
    \sum_{\substack{x_1,\dots, x_n \in \delta\Z^2\cap [0,1]^2 \\ \inf_{i\neq j}|x_i-x_j|<\eps}} (\delta^2\varrho(\delta)^{-\frac{1}{2}})^n \E_\delta\Ll[ \S_{x_1}\cdots\S_{x_n}\Rr] \leq C\eps^c .
\end{align*}
Since $h$ is bounded, this ensures that the term in~\eqref{e.<Phih>^n=_II} is bounded by $C\|h\|^n_\infty\eps^c$. First taking $\delta\to0$ and then $\eps\to0$, we can deduce the desired result.
\end{proof}

We recall the following standard form of \textit{Carleman's condition} (e.g.\ see~\cite[Theorem 3.3.25 and Remark]{durrett2019probability}). Let
$(m_n)_{n\in\N}$ be the moment sequence of a probability measure on
$\mathbb R$. If $\sum_{k=1}^\infty m_{2k}^{-1/(2k)}=\infty$, then the
probability measure is uniquely determined by its moments. 

\begin{lemma}
\label{l.carleman_Phi}
Let $s>1$ and let $h\in\mathcal H^s_0(\mathbb R^2)$. 
Then, there is a constant $A_h>0$ such that, for every $k\in\N$, we have $|M_{2k}(h)|\leq A_h^{2k}(2k)!$.
Consequently, the sequence $(M_n(h))_{n\in\N}$ satisfies Carleman's condition.
\end{lemma}

\begin{proof}
By \eqref{e.t}, for $n=2k$, we have
\begin{align}
    \langle x_1,\dots,x_{2k}\rangle
    =
    \sum_{\substack{\xi\in\{-1,+1\}^{2k}\\ \sum_i \xi_i=0}}
    \prod_{1\leq i<j\leq 2k}|x_i-x_j|^{\xi_i\xi_j/4}. \label{e.<>2k}
\end{align}
Fix any $\xi$ satisfying $\sum_i\xi_i=0$ and set $P=\{i:\xi_i=+1\}$ and $N=\{i:\xi_i=-1\}$.
Then, we have $|P|=|N|=k$ and
\[
    \prod_{1\leq i<j\leq 2k}|x_i-x_j|^{\xi_i\xi_j/4}
    =
    \frac{
    \prod_{\substack{i<j\\ i,j\in P}}|x_i-x_j|^{1/4}
    \prod_{\substack{i<j\\ i,j\in N}}|x_i-x_j|^{1/4}
    }{
    \prod_{\substack{i\in P\\ j\in N}}|x_i-x_j|^{1/4}
    }.
\]
By the Cauchy double alternant identity\footnote{See
\cite[(0.9.12.1) in Sec.~0.9.12]{hornjohnson2013matrix} of the form
\[
    \det\left(\frac{1}{a_r+b_s}\right)_{1\leq r,s\leq k}
    =
    \frac{
    \prod_{r<r'}(a_{r'}-a_r)
    \prod_{s<s'}(b_{s'}-b_s)
    }{
    \prod_{r,s}(a_r+b_s)
    }.
\]
Replacing \(b_s\) by \(-b_s\), and then taking absolute values, gives the
identity used here with \(a_r=x_{p_r}\), \(b_s=x_{n_s}\), and $\R^2$ identified with $\mathbb{C}$, where
\(P=\{p_1,\dots,p_k\}\) and \(N=\{n_1,\dots,n_k\}\).}, we have (viewing $x_i$'s as complex numbers on the right-hand side) that
\[
    \frac{
    \prod_{\substack{i<j\\ i,j\in P}}|x_i-x_j|
    \prod_{\substack{i<j\\ i,j\in N}}|x_i-x_j|
    }{
    \prod_{\substack{i\in P\\ j\in N}}|x_i-x_j|
    }
    =
    \left|
    \det\left(\frac{1}{x_i-x_j}\right)_{i\in P,\ j\in N}
    \right|.
\]
Therefore, expanding the determinant and using subadditivity of $u\mapsto u^{1/4}$,
we get
\[
    \prod_{1\leq i<j\leq 2k}|x_i-x_j|^{\xi_i\xi_j/4}
    \leq
    \sum_{\pi:P\to N}
    \prod_{i\in P}|x_i-x_{\pi(i)}|^{-1/4},
\]
where the sum is over all bijections $\pi:P\to N$.
Combining the above three displays, we obtain
\[
\begin{aligned}
&\int_{(\mathbb R^2)^{2k}}
\prod_{\ell=1}^{2k}|h(x_\ell)|
\prod_{1\leq i<j\leq 2k}|x_i-x_j|^{\xi_i\xi_j/4}
\,dx_1\cdots dx_{2k}
\\
&\qquad\leq
\sum_{\pi:P\to N}
\int_{(\mathbb R^2)^{2k}}
\prod_{\ell=1}^{2k}|h(x_\ell)|
\prod_{i\in P}|x_i-x_{\pi(i)}|^{-1/4}
\,dx_1\cdots dx_{2k}=
k!\,B_h^k
\end{aligned}
\]
where $B_h:=\iint_{\mathbb R^2\times\mathbb R^2}|h(x)|\,|h(y)|\,|x-y|^{-1/4}\,dx\,dy$. Since $h$ has compact support and $s>1$, $h$ is bounded. Moreover,
$|x-y|^{-1/4}$ is locally integrable in two dimensions, so $B_h<\infty$.
Since $\#\{\xi\in\{\pm1\}^{2k}:\:\sum_i\xi_i=0\}=\binom{2k}{k}$, in view of~\eqref{e.<>2k} and $M_{2k}(h)$ given in~\eqref{e.M_n(h)=}, we have
\[
\begin{aligned}
    |M_{2k}(h)|
    &\leq
    \binom{2k}{k}
    k!\,B_h^k
    =
    \frac{(2k)!}{k!}
    B_h^k
    \leq
    A_h^{2k}(2k)!
\end{aligned}
\]
for some constant $A_h$.
By Stirling's formula, we have $((2k)!)^{-1/(2k)} \asymp \frac1k$ and thus
\[
    M_{2k}(h)^{-1/(2k)}
    \geq
    A_h^{-1}((2k)!)^{-1/(2k)}
    \geq
    \frac{c_h}{k}
s\]
for some $c_h>0$. Now, we can conclude $\sum_{k=1}^\infty M_{2k}(h)^{-1/(2k)} = \infty$, which is Carleman's condition.
\end{proof}

\begin{proof}[Proof of Theorem~\ref{t.Phi}]
For brevity, we write $\stackrel{\d}{\to}$ to denote convergence in law along a sequence and in a space that will be clear from the context.
Fix $s>1$. 
Let $(\delta_j)_{j\in\N}$ be any sequence with $\delta_j\to0$. By tightness in Lemma~\ref{l.tightness}, after
passing to a subsequence, still denoted by $(\delta_j)_{j\in\N}$, there exists a random
distribution $\Phi$ such that $\Phi_{\delta_j}\stackrel{\d}{\to} \Phi$ in $\mathcal H^{-s}(\mathbb R^2)$.
Let us identify the law of this arbitrary subsequential limit.

Fix any $h\in\mathcal H^s_0(\mathbb R^2)$ and set $Y_j:=\int \Phi_{\delta_j}h$ and $Y:=\int \Phi h$. 
Since $u\mapsto \int uh$ is continuous on $\mathcal H^{-s}(\mathbb R^2)$, we have $Y_j\stackrel{\d}{\to}Y$.
We claim that $Y$ has the limiting moments from Lemma~\ref{l.mom_Phi_delta}.
Let $p\in\N$ and choose an even integer $q>p$. 
Lemma~\ref{l.mom_Phi_delta} ensures that $\sup_j \E_{\delta_j}\Ll[|Y_j|^q \Rr]<\infty$ and thus $(Y_j^p)_j$ is uniformly integrable. Since $Y_j\stackrel{\d}{\to} Y$, we get $\E\Ll[ Y^p\Rr]=\lim_{j\to\infty}\E_{\delta_j}\Ll[ Y_j^p\Rr]$.
Applying Lemma~\ref{l.mom_Phi_delta} again gives
\begin{align}\label{e.E[()^p]=M_p(h)}
    \E\Ll[
    \left(\int \Phi h\right)^p
    \Rr]
    =
    M_p(h).
\end{align}
By Lemma~\ref{l.carleman_Phi}, the law of $\int\Phi h$ is uniquely
determined by these moments and thus independent of
the subsequential limit.

It remains to show that the law of $\Phi$ is independent of the subsequence. Let $d\in\N$,
$h_1,\ldots,h_d\in\mcl H^s_0(\R^2)$, and $a\in\R^d$. By the above
the law of
$\sum_{i=1}^d a_i\int\Phi h_i=\int\Phi(\sum_{i=1}^d a_i h_i)$ is independent
of the subsequence. By Cramér--Wold, the joint law of
$(\int\Phi h_1,\ldots,\int\Phi h_d)$ is independent of the subsequence. Since
the maps $u\mapsto\int uh$, for $h\in\mcl H^s_0(\R^2)$, are continuous, separate
points, and generate the Borel $\sigma$-algebra associated with the local
$\mcl H^{-s}$-topology, all subsequential limits have the same law.

Finally, tightness in \(\mcl H^{-s}(\R^2)\), endowed with this local topology,
together with uniqueness of subsequential limits implies $\Phi_\delta \stackrel{\d}{\to} \Phi_0$ in $\mcl H^{-s}(\R^2)$.
The moment formula for \(\Phi_0\) follows from~\eqref{e.E[()^p]=M_p(h)}.
Moreover, by Lemma~\ref{l.carleman_Phi}, the law of \(\int \Phi_0 h\) is
uniquely determined by these moments.
\end{proof}

\appendix

\section{Choice of the slowly varying sequence \texorpdfstring{$\eps(\delta)$}{ε(δ)}}

The following procedure of choosing $\eps(\delta)$ that decays slowly as $\delta\to0$ was used in the proof of Proposition~\ref{p}.

\begin{lemma}\label{lem:choose_epsilon}
Fix $n\in2\N$ and $d\in(0,1)$. Let $\rho$ be a function satisfying $\rho(t)\to0$ as $t\downarrow0$. There exists
$\eps(\delta)\in(0,d/2)$ such that we have, as $\delta\to0$,
\begin{align*}
    \eps(\delta)\to0,\qquad \tfrac{\delta}{\eps(\delta)}\to0,\qquad
    \eps(\delta)^{\rho(\delta/\eps(\delta))}\to1,
\end{align*}
and, for every $\tau\in4\Z\cap[-n,n]$ and every $\xi\in\Xi_\tau$, the convergence in
\emph{\eqref{e.GFF_compute}} holds with $\eps$ replaced by $\eps(\delta)$, uniformly in $x\in X(n,d)$.
\end{lemma}

\begin{proof}
Choose $\eps_k\downarrow0$ as $k\to\infty$ with $\eps_k<d/2$. For $\xi\in\Xi_\tau$,
$\tau\in4\Z\cap[-n,n]$, set
\begin{align*}
    H_{\xi,\eps}(x)=
    \begin{cases}
    e^{nc_0}\prod_{1\le i<j\le n}
    \Ll(|x_i-x_j|^{\xi_i\xi_j/4}
    e^{2\xi_i\xi_jc_\eps(x_i-x_j)}\Rr),
    & \xi\in\Xi_0,\\
    0, & \xi\in\Xi_\tau,\ \tau\neq0,
    \end{cases}
\end{align*}
and
\begin{align*}
    \Delta_\xi(\delta,\eps)
    =
    \sup_{x\in X(n,d)}
    \Ll|
    \eps^{-n/8}\phi_{\delta\Z^2}
    \Ll[\mcl A_{F_{x,\xi,\eps}}(\mcl L)\Rr]
    -H_{\xi,\eps}(x)
    \Rr|.
\end{align*}
By Lemmas~\ref{l.GFF_test_function} and~\ref{l.e.GFF_compute}, for every fixed $k$,
$\Delta_\xi(\delta,\eps_k)\to0$ as $\delta\to0$. Since there are only finitely many
$\xi$'s, we may choose $\delta_k\downarrow0$ such that, for every $\delta<\delta_k$,
\begin{align*}
    \max_{\tau\in4\Z\cap[-n,n]}\max_{\xi\in\Xi_\tau}
    \Delta_\xi(\delta,\eps_k)\le \tfrac1k .
\end{align*}
We choose $\delta_k$ smaller if necessary so that
\begin{align*}
    \tfrac{\delta_k}{\eps_k}\le t_k
    \qquad\text{and}\qquad
    \sup_{0<t<t_k}\rho(t)\log\tfrac1{\eps_k}\le \tfrac1k ,
\end{align*}
for some $t_k\downarrow0$. Define $\eps(\delta)=\eps_k$ when
$\delta_{k+1}\le\delta<\delta_k$. Then $\eps(\delta)\to0$,
$\delta/\eps(\delta)\to0$, and
\begin{align*}
    \rho\Ll(\tfrac{\delta}{\eps(\delta)}\Rr)
    \log\tfrac1{\eps(\delta)}\to0,
\end{align*}
which gives $\eps(\delta)^{\rho(\delta/\eps(\delta))}\to1$. The definition of
$\delta_k$ gives the desired diagonal convergence. Finally, when $\xi\in\Xi_0$,
the factor $e^{2\xi_i\xi_jc_{\eps(\delta)}(x_i-x_j)}$ tends to $1$ uniformly on
$X(n,d)$, due to $|x_i-x_j|\ge d$ (see~\eqref{e.|c_eps(v)|<}). This proves the lemma.
\end{proof}

\section{Uniform convergence of test functions supported on \texorpdfstring{$\eps$-balls}{ε-balls}}

We verify the uniformity used in Lemmas~\ref{l.GFF_test_function} and~\ref{l.e.GFF_compute}. The statement below is slightly stronger than what is needed there. In those lemmas, $\eps$ and $\xi$ are fixed, and only uniformity in $x$ is required.

\begin{lemma}\label{lem:uniform_gff_convergence}
Let $n\in\N$.
For $r>1$, let $\Delta(r)$ be the collection of parameters $(x,\xi,\eps)=((x_i)_{i\in\llbracket 1,n\rrbracket},(\xi_i)_{i\in\llbracket 1,n\rrbracket},\eps)\in(\R^2)^n\times \R^n\times (0,\infty)$ satisfying $\max_{i\in\llbracket 1,n\rrbracket}|x_i|\leq r$, $\max_{i\in\llbracket 1,n\rrbracket}|\xi_i|\leq r$, $\sum_{i=1}^n\xi_i=0$, and $\eps\in[\frac{1}{r},r]$. Given $(x,\xi,\eps)$, we consider the test function $F_{x,\xi,\eps} = \sum_{i=1}^n \frac{\xi_i}{2\eps^2}\mathbf{1}_{B_\eps(x_i)}$.
Then, for every $r>1$, we have
\begin{align*}
    \lim_{\delta\to0}\sup_{(x,\xi,\eps)\in\Delta(r)}\Ll|\phi_{\delta\Z^2}\Ll[\mathcal{A}_{F_{x,\xi,\eps}}(\mcl L)\Rr]- \exp\left( \frac{1}{2\pi^2}\iint F_{x,\xi,\eps}(\rz)F_{x,\xi,\eps}(\rz')\log|\rz-\rz'|\d\rz\d\rz'\right)\Rr|=0.
\end{align*}
\end{lemma}

\begin{proof}
Let $h_\delta$ be the six-vertex configuration coupled with the random cluster model $\phi_{\delta\Z^2}$ via the Baxter--Kelland--Wu correspondence (see \cite{baxter1976equivalence,DumManTas16b,magicformula}), which gives $\phi_{\delta\Z^2}\Ll[\mathcal{A}_{F_{x,\xi,\eps}}(\mcl L)\Rr] = \E\Ll[e^{\mathrm{i}\int F_{x,\xi,\eps}h_\delta}\Rr]$ where $h_\delta$ is viewed as a random distribution on $\R^2$ in a natural way. Then, the exponential term in the display is equal to $\E\Ll[e^{\mathrm{i}\int F_{x,\xi,\eps} h}\Rr]$ where $h$ is the planar GFF with variance $\frac{2}{\pi}$. 
Hence, it is equivalent to showing
\begin{align*}
    \lim_{\delta\to0}\sup_{(x,\xi,\eps)\in\Delta(r)}\Ll|\E\Ll[e^{\mathrm{i}\int F_{x,\xi,\eps}h_\delta}\Rr]- \E\Ll[e^{\mathrm{i}\int F_{x,\xi,\eps} h}\Rr]\Rr|=0.
\end{align*}
Choose $L>0$ large enough so that $\supp F_{x,\xi,\eps}\subset(-L,L)^2$.
Fix any $s\in(0,1/2)$ and we have $F_{x,\xi,\eps} \in V:= \mcl H^s_0([-L,L]^2)$ (see the discussion below~\eqref{e.Xi_tau=}).
Restrictions of $h$ and $h_\delta$ to $[-L,L]^2$ are random distributions in $V^* := H^{-s}([-L,L]^2)$. 

Define $\Psi : \Delta(r) \to V$ by $\Psi(x, \xi, \eps) := F_{x,\xi,\eps}$. From the expression of $F_{x,\xi,\eps}$, we can verify that $\Psi$ is continuous. Thus, $\mathcal{F} := \Psi(\Delta(r))$ is compact in $V$ as $\Delta(r)$ is compact.
For any $f\in \mathcal F$, write $\varphi_\delta(f) := \E[e^{\mathrm{i}\int f\, h_\delta}]$ and $\varphi(f) := \E[e^{\mathrm{i}\int f\, h}]$. By~\cite[Theorem~2.8 and Definition~2.7~(iii)]{GFF6vertex}, $h_\delta$ converges to $h$ in law in $V^*$, which implies $\varphi_\delta(f) \to \varphi(f)$ pointwise for every fixed $f \in V$.

We now show that the family $\{\varphi_\delta\}_{\delta>0}$ is equicontinuous on $V$. 
By Prokhorov's Theorem, the convergence in law of $h_\delta$ implies that $\{h_\delta\}_{\delta>0}$ is tight. Hence, for any $\eta > 0$, there is a compact set $K \subset V^*$ such that $\sup_{\delta > 0} \mathbb{P}(h_\delta \notin K) < \eta$.
Because $K$ is compact in $V^*$, it is bounded; hence, there exists a constant $M > 0$ such that $\sup_{u \in K} \|u\|_{V^*} \leq M$. 
For any two test functions $f, g \in V$, we have
\begin{align*}
    &|\varphi_\delta(f) - \varphi_\delta(g)| 
    = \Ll| \E\Ll[ e^{\mathrm{i}\int f\,h_\delta} - e^{\mathrm{i}\int g\, h_\delta} \Rr] \Rr|\leq \E\Ll[ \Ll| e^{\mathrm{i}\int(f-g) \,h_\delta} - 1 \Rr| \Rr] 
    \\
    &= \E\Ll[ \Ll| e^{\mathrm{i}\int(f-g) \,h_\delta} - 1 \Rr| \Ll(\one_{\{h_\delta \in K\}}+\one_{\{h_\delta \notin K\}}\Rr)\Rr]
    \leq \E\Ll[ \Ll|\int (f-g)\, h_\delta\Rr| \1_{\{h_\delta \in K\}} \Rr] + 2 \mathbb{P}(h_\delta \notin K).
\end{align*}
The Cauchy-Schwarz inequality gives
$\Ll|\int (f-g)\, h_\delta\Rr| \leq \|f-g\|_V \|h_\delta\|_{V^*}$. On the event $\{h_\delta \in K\}$, we have $\|h_\delta\|_{V^*} \leq M$, which together with the choice of $K$ gives
\begin{align*}
    |\varphi_\delta(f) - \varphi_\delta(g)| 
    &\leq  M \|f-g\|_V + 2\eta.
\end{align*}
This implies the equicontinuity of $\{\varphi_\delta\}_{\delta>0}$ on $V$.

Since $\varphi_\delta$ converges to $\varphi$ pointwise, $\{\varphi_\delta\}_{\delta>0}$ is equicontinuous, and $\mathcal{F}$ is compact, the Arzelà-Ascoli theorem yields $\lim_{\delta\to0} \sup_{f \in \mathcal{F}} \Ll| \varphi_\delta(f) - \varphi(f) \Rr| = 0$.
Recalling that $\mathcal{F} = \{F_{x,\xi,\eps} : (x,\xi,\eps)\in\Delta(r)\}$, this uniform convergence over $\mathcal{F}$ is exactly the stated limit. This concludes the proof.
\end{proof}

\newpage

\noindent \textbf{Acknowledgements.} Hong-Bin Chen acknowledges funding from the NYU Shanghai Start-Up Fund and support from the NYU–ECNU Institute of Mathematical Sciences at NYU Shanghai. ChatGPT was used only for writing improvement and proofreading; its sole mathematical input was suggesting Carleman's condition in the proof of Theorem~\ref{t.Phi} to identify the law from its moments.

\bibliographystyle{plain}

\begin{thebibliography}{10}

\bibitem{alcalde2026ising}
Tom{\'a}s Alcalde~L{\'o}pez, Lorca Heeney, and Marcin Lis.
\newblock The {Ising} magnetisation field and the {Gaussian} free field.
\newblock {\em arXiv preprint arXiv:2602.05886}, 2026.

\bibitem{alcalde2026excursion}
Tom{\'a}s Alcalde~L{\'o}pez and Avelio Sep{\'u}lveda.
\newblock Excursion decomposition of the {XOR-Ising} model.
\newblock {\em arXiv preprint arXiv:2602.06011}, 2026.

\bibitem{aru2025noiselike}
Juhan Aru, Guillaume Baverez, Antoine Jego, and Janne Junnila.
\newblock Noise-like analytic properties of imaginary chaos.
\newblock {\em Electronic Journal of Probability}, 30(16):1--43, 2025.

\bibitem{magicformula}
E.~Averous, H.-B. Chen, H.~Duminil-Copin, T.~He, D.~Krachun, I.~Manolescu, and
  J.~Xia.
\newblock Rotational invariance of any six-vertex model scaling limit with
  $-1\le\delta\le-\frac12$.
\newblock Manuscript in preparation.

\bibitem{baxter1976equivalence}
Rodney~J Baxter, Stewart~B Kelland, and Frank~Y Wu.
\newblock Equivalence of the potts model or whitney polynomial with an ice-type
  model.
\newblock {\em Journal of Physics A: Mathematical and General}, 9(3):397, 1976.

\bibitem{Beffara08}
Vincent Beffara.
\newblock The dimension of the sle curves.
\newblock {\em Annals of Probability}, 36(4):1421--1452, 2008.

\bibitem{beffara2012self}
Vincent Beffara and Hugo Duminil-Copin.
\newblock The self-dual point of the two-dimensional random-cluster model is
  critical for $q\geq1$.
\newblock {\em Probability Theory and Related Fields}, 153(3):511--542, 2012.

\bibitem{cai2025three}
Gefei Cai, Haoyu Liu, Baojun Wu, and Zijie Zhuang.
\newblock {Three-point connectivity constant for $q$-state Potts spin
  clusters}.
\newblock {\em arXiv preprint arXiv:2510.05850}, 2025.

\bibitem{camia2024conformal}
Federico Camia.
\newblock Conformal covariance of connection probabilities and fields in 2d
  critical percolation.
\newblock {\em Communications on Pure and Applied Mathematics},
  77(3):2138--2176, 2024.

\bibitem{camia2024conformally}
Federico Camia and Yu~Feng.
\newblock Conformally covariant probabilities, operator product expansions, and
  logarithmic correlations in two-dimensional critical percolation.
\newblock {\em arXiv preprint arXiv:2407.04246}, 2024.

\bibitem{camia2024logarithmic}
Federico Camia and Yu~Feng.
\newblock Logarithmic correlation functions in 2d critical percolation.
\newblock {\em Journal of High Energy Physics}, 2024(8):1--25, 2024.

\bibitem{camia2025}
Federico Camia and Yu~Feng.
\newblock Conformal covariance of connection probabilities in the 2d critical
  fk-ising model.
\newblock {\em Stochastic Processes and their Applications}, 189:104734, 06
  2025.

\bibitem{camia2015planar}
Federico Camia, Christophe Garban, and Charles~M Newman.
\newblock Planar ising magnetization field i. uniqueness of the critical
  scaling limit.
\newblock {\em The Annals of Probability}, pages 528--571, 2015.

\bibitem{cardy1980scaling}
John~L Cardy, M~Nauenberg, and DJ~Scalapino.
\newblock {Scaling theory of the Potts-model multicritical point}.
\newblock {\em Physical Review B}, 22(5):2560, 1980.

\bibitem{chelkak2021correlations}
Dmitry Chelkak, Cl{\'e}ment Hongler, and Konstantin Izyurov.
\newblock Correlations of primary fields in the critical ising model.
\newblock {\em arXiv preprint arXiv:2103.10263}, 2021.

\bibitem{CHI2015}
Dmitry Chelkak, Clément Hongler, and Konstantin Izyurov.
\newblock Conformal invariance of spin correlations in the planar ising model.
\newblock {\em Annals of Mathematics}, 181(3):1087--1138, 2015.

\bibitem{qequal4}
H.-B. Chen, H.~Duminil-Copin, T.~He, F.~Jacopin, D.~Krachun, I.~Manolescu, and
  J.~Xia.
\newblock Critical exponents for the planar random-cluster model with
  cluster-weight $q=4$.
\newblock {\em arXiv preprint arXiv:2605.30030}, 2026.

\bibitem{den1983extended}
Marcel den Nijs.
\newblock {Extended scaling relations for the magnetic critical exponents of
  the Potts model}.
\newblock {\em Physical Review B}, 27(3):1674, 1983.

\bibitem{dotsenko1984conformal}
Vladimir~S Dotsenko and Vladimir~A Fateev.
\newblock Conformal algebra and multipoint correlation functions in {2D}
  statistical models.
\newblock {\em Nuclear Physics B}, 240(3):312--348, 1984.

\bibitem{DumManTas16b}
Hugo Duminil{-}Copin, Maxime Gagnebin, Matan Harel, Ioan Manolescu, and Vincent
  Tassion.
\newblock Discontinuity of the phase transition for the planar random-cluster
  and potts models with $q>4$.
\newblock {\em Annales scientifiques de l'{\'E}cole Normale Sup{\'e}rieure},
  54:1363--1413, 2021.

\bibitem{GFF6vertex}
Hugo Duminil-Copin, Karol Kozlowski, Piet Lammers, and Ioan Manolescu.
\newblock {Gaussian free field convergence of the six-vertex model with
  $-1\leq\Delta\leq-\frac{1}{2}$}.
\newblock {\em arXiv preprint arXiv:2603.06268}, 2026.

\bibitem{duminil2020rotational}
Hugo Duminil-Copin, Karol~Kajetan Kozlowski, Dmitry Krachun, Ioan Manolescu,
  and Mendes Oulamara.
\newblock Rotational invariance in critical planar lattice models.
\newblock {\em arXiv preprint arXiv:2012.11672}, 2020.

\bibitem{duminil2022planar}
Hugo Duminil-Copin and Ioan Manolescu.
\newblock Planar random-cluster model: scaling relations.
\newblock {\em Forum of Mathematics, Pi}, 10:e23, 2022.

\bibitem{DumSidTas13}
Hugo Duminil-Copin, Vladas Sidoravicius, and Vincent Tassion.
\newblock Continuity of the phase transition for planar random-cluster and
  {P}otts models with $1\leq q\leq 4$.
\newblock {\em Communications in Mathematical Physics}, 349(1):47--107, 2017.

\bibitem{durrett2019probability}
Rick Durrett.
\newblock {\em Probability: Theory and Examples}.
\newblock Number~49 in Cambridge Series in Statistical and Probabilistic
  Mathematics. Cambridge University Press, Cambridge, 5 edition, 2019.

\bibitem{garban2013pivotal}
Christophe Garban, G{\'a}bor Pete, and Oded Schramm.
\newblock Pivotal, cluster, and interface measures for critical planar
  percolation.
\newblock {\em Journal of the American Mathematical Society}, 26(4):939--1024,
  2013.

\bibitem{hornjohnson2013matrix}
Roger~A. Horn and Charles~R. Johnson.
\newblock {\em Matrix Analysis}.
\newblock Cambridge University Press, Cambridge, 2 edition, 2013.

\bibitem{kadanoff1979correlation}
Leo~P Kadanoff and Alan~C Brown.
\newblock Correlation functions on the critical lines of the {Baxter} and
  {Ashkin-Teller} models.
\newblock {\em Annals of Physics}, 121(1-2):318--345, 1979.

\bibitem{kohler2025fuzzy}
Laurin K{\"o}hler-Schindler and Matthis Lehmkuehler.
\newblock {The fuzzy Potts model in the plane: scaling limits and arm
  exponents}.
\newblock {\em Probability Theory and Related Fields}, 191(1):287--359, 2025.

\bibitem{LSW2004}
Gregory~F. Lawler, Oded Schramm, and Wendelin Werner.
\newblock Conformal invariance of planar loop-erased random walks and uniform
  spanning trees.
\newblock {\em Annals of Probability}, 32(1B):939--995, 2004.

\bibitem{liu2024bulk}
Haoyu Liu, Xin Sun, Pu~Yu, and Zijie Zhuang.
\newblock {The bulk one-arm exponent for the $\mathrm{CLE}_{\kappa'}$
  percolations}.
\newblock {\em arXiv preprint arXiv:2410.12724}, 2024.

\bibitem{MillerSheffield16I}
Jason Miller and Scott Sheffield.
\newblock {Imaginary Geometry I: Interacting SLEs}.
\newblock {\em Probability Theory and Related Fields}, 164:553--705, 2016.

\bibitem{MillerSheffield16II}
Jason Miller and Scott Sheffield.
\newblock {Imaginary Geometry II: Reversibility of SLE${}_\kappa(\rho)$ for
  $\kappa\in(0,4)$}.
\newblock {\em Annals of Probability}, 44(3):1647--1722, 2016.

\bibitem{MillerSheffield16III}
Jason Miller and Scott Sheffield.
\newblock {Imaginary Geometry III: Conformal welding}.
\newblock {\em Probability Theory and Related Fields}, 166:553--626, 2016.

\bibitem{newman2019leeyang}
Charles~M. Newman and Wei Wu.
\newblock {Lee--Yang} property and {Gaussian} multiplicative chaos.
\newblock {\em Communications in Mathematical Physics}, 369(1):153--170, 2019.

\bibitem{Onsager44}
Lars Onsager.
\newblock {Crystal Statistics. I. A Two-Dimensional Model with an
  Order–Disorder Transition}.
\newblock {\em Physical Review}, 65(3-4):117--149, 1944.

\bibitem{salas1997logarithmic}
Jesus Salas and Alan~D Sokal.
\newblock {Logarithmic corrections and finite-size scaling in the
  two-dimensional 4-state Potts model}.
\newblock {\em Journal of statistical physics}, 88(3):567--615, 1997.

\bibitem{sickel2020regularity}
Winfried Sickel.
\newblock On the regularity of characteristic functions.
\newblock In {\em Anomalies in Partial Differential Equations}, pages 395--441.
  Springer, 2020.

\bibitem{Smirnov10}
Stanislav Smirnov.
\newblock {Conformal invariance in random cluster models. I. Holomorphic
  fermions in the Ising model}.
\newblock {\em Annals of Mathematics}, 172(2):1435--1467, 2010.

\bibitem{WuMcCoyTracyBarouch76}
Tai~Tsun Wu, Barry~M. McCoy, Craig~A. Tracy, and Elie Barouch.
\newblock {Spin–spin correlation functions for the two-dimensional Ising
  model: Exact theory in the scaling region}.
\newblock {\em Physical Review B}, 13(1):316--374, 1976.

\bibitem{zamolodchikov1985two}
Alexander~B Zamolodchikov.
\newblock Two-dimensional conformal symmetry and critical four-spin correlation
  functions in the {Ashkin-Teller} model.
\newblock {\em Soviet Physics JETP}, 62(4):653--659, 1985.

\end{thebibliography}
\newcommand{\noop}[1]{} \def\cprime{$'$}

\end{document}